\documentclass{article}
\usepackage{graphicx}
\usepackage{amsfonts}
\usepackage{amsthm}

\newtheorem{lem}{Lemma}
\newtheorem{thm}{Theorem}
\theoremstyle{definition}
\newtheorem{conj}{Conjecture}
\newtheorem{defn}{Definition}
\newtheorem{remark}{Remark}

\title{New horizons in multidimensional diffusion:\\
The Lorentz gas and the Riemann Hypothesis}
\author{Carl P. Dettmann, School of Mathematics, University of Bristol,
Bristol BS81TW, UK}
\date{\today}

\begin{document}
\maketitle

\begin{abstract}
The Lorentz gas is a billiard model involving a point particle diffusing
deterministically in a periodic array of convex scatterers.  In the two
dimensional finite horizon case, in which all trajectories involve
collisions with the scatterers, displacements scaled by the usual
diffusive factor $\sqrt{t}$ are normally distributed, as shown by Bunimovich
and Sinai in 1981.  In the infinite horizon case, motion is superdiffusive,
however the normal distribution is recovered when scaling by $\sqrt{t\ln t}$,
with an explicit formula for its variance. Here we explore the infinite horizon
case in arbitrary dimensions, giving explicit formulas for the mean square
displacement, arguing that it differs from the variance of the limiting
distribution, making connections with the Riemann Hypothesis in the small
scatterer limit, and providing evidence for a critical dimension $d=6$
beyond which correlation decay exhibits fractional powers.  The results are
conditional on a number of conjectures, and are corroborated by numerical
simulations in up to ten dimensions.
\end{abstract}

\section{Introduction}
The Lorentz gas is a billiard model~\cite{CM}, in which a point particle moves
in straight lines with unit speed except for specular collisions with hard scatterers, most
often disks or balls.  It is an ``extended'' billiard in that it moves in an unbounded domain,
in contrast to the usual billiard setting.  Lorentz developed an approach along these lines
to study electrical conduction
in 1905~\cite{L}, however periodic arrays of scatterers in the plane or in
higher dimensions appeared as early as 1873 (the Galton board, or
``quincunx''~\cite{G}). They have been a major example of hyperbolic dynamics
with singularities starting from the work of Sinai in 1970~\cite{Si}.

Lorentz gas models are considered in the physics literature, for
example arising as the molecular dynamics model with two atoms and a hard
potential after reduction
of the centre of mass motion~\cite{D}.  It is thus probably the simplest model
of deterministic diffusion with physical attributes such as a Hamiltonian
structure, continuous time dynamics and time reversal invariance; simpler but
less physical models include one dimensional maps on
$f:\mathbb{R}\to\mathbb{R}$ with a periodicity property $f(x+1)=f(x)+1$
~\cite{FG,SFK} and multibaker maps on $[0,1]^2\times\mathbb{Z}$~\cite{Ga}.  In
the former case, with $f$ a piecewise linear map, the diffusion constant is a
non-smooth function of parameters, as quantified in~\cite{KHK}. It would
be interesting to extend such results to billiards; numerical studies of Lorentz(-like)
models have suggested an irregular but somewhat smoother behaviour than the piecewise
linear map~\cite{HG01,HKG02,KD00,KK02}.

There is a substantial literature on diffusion in periodic Lorentz gases
in dimension $d=2$ in the setting of dispersing billiards, for example 
sufficiently smooth ($C^3$), convex with curvature bounded away from zero,
and an upper bound on the free flight length between scatterers (`finite horizon').
It is known that the displacements of particles initially distributed
uniformly (outside the scatterers) converges weakly to a Wiener process (Brownian motion)
with variance proportional to the (continuous) time $t$ as $t\to\infty$~\cite{BS};
proof of recurrence in the full space followed later~\cite{Conze,Sc}.
The
strategy is usually to relate the deterministic model to a Markovian
random walk process for which the results are known or easier to prove.

If we relax the finite horizon condition, the distribution remains Gaussian, but with a $t\ln t$
scaling.  The two dimensional case was considered by Zacherl et al.~\cite{ZGNR86} who gave
approximate formulas for the velocity correlation functions, and hence diffusion properties.
Bleher~\cite{B}, proved some of the necessary results, deriving some exact coefficients.
Detailed numerical results for velocity correlations were obtained by
Matsuoka and Martin~\cite{MM97}.  Sz\'asz and Varju~\cite{SV}, completed the proofs
in discrete time ($n$), obtaining a local limit law to a normal distribution, and
also recurrence and ergodicity in the full space.  Techniques there were based
on work of B\'alint and Gou\"ezel~\cite{BG} for the stadium billiard, which also
has long flights without a collision on a curved boundary, and also a normal
distribution with a nonstandard $n\ln n$ limit law.  Finally, Chernov and
Dolgopyat~\cite{CD} proved weak convergence to a Wiener distribution in
the continuous time case.

In general it is
easier to prove statements about correlation decay and diffusion for the
discrete time dynamics first; some diffusion results can then be transferred
to continuous time, as discussed in Ref.~\cite{B,CD} and Section~\ref{s:disc} below.
Correlation decay for billiards in continuous time have been
considered~\cite{BM,C07,M} but are in general more difficult.  Here we use
continuous time $(t)$ with connections to the discrete time diffusion
process discussed in Sec.~\ref{s:disc}.

The main purpose of this paper is to explore the infinite horizon case in
dimensions $d>2$, or more precisely, several varieties of infinite horizon 
that appear in higher
dimensions.  Recurrence is of course not to be expected, even for finite 
horizon,
however a (super)-diffusion coefficient and further refinements may still be
considered.  One case in which the diffusion in higher dimensional Lorentz gases has
been treated very recently is the Boltzmann-Grad limit~\cite{MS1,MS2}.  Here the
radius and density of scatterers scale so that the mean free path is constant.
Here we will also consider the limit of small scatterers, but only after the
infinite time limit.  We will nevertheless be able to make connection with an asymptotic
formula in~\cite{MS2}.  In addition, we will find connections
with the Riemann Hypothesis, along the lines of those presented for the escape
problem of a billiard inside a circle with one or two holes~\cite{BD}; this is also new for $d=2$.
Last but not least,
there are some outstanding issues in $d=2$ regarding
the diffusion coefficients defined using moments and using
a limiting normal distribution, which we need to clarify
first (Sec.~\ref{s:DXi}); these may also have wider implications for our understanding of
other anomalous diffusion processes.  Main results are stated and discussed in
Sec.~\ref{s:summary}, which concludes with a description of the following
more detailed sections.

\section{Diffusion coefficients}\label{s:DXi}
Different conventions are used in physical and mathematical literature for
diffusion, and there are further subtleties involved with the
definition of the logarithmic superdiffusion coefficient.
Fick's law, dating back to 1855, states that the flux of
particles in a diffusing species is proportional to the
concentration gradient.  If, as in the Lorentz gas, the diffusing
particles do not interact with each other, and the medium
through which they diffuse (periodic collection of
scatterers) is homogeneous at large scales but not necessarily isotropic,
from physical considerations we expect a macroscopic
equation of the form
\begin{equation}\label{e:diffeq}
\frac{\partial\rho}{\partial t}=D_{ij}\frac{\partial}{\partial \tilde x_i}\frac{\partial}{\partial \tilde x_j}\rho
\end{equation}
where $\rho(\tilde{\bf x},t)$ is the density of particles, $i$ and $j$ are spatial indices, we
are using the Einstein summation convention (summing repeated indices from $1$ to $d$), and we
denote quantities in the full (unbounded) space with tildes.  The $D_{ij}$ are components of the diffusivity matrix $\bf D$, which is positive semi-definite and symmetric.  For now we assume
that it is actually positive definite, corresponding
to the possibility of diffusion in all directions. In simple cases it is often isotropic, $D_{ij}=D\delta_{ij}$ where $D$ is ``the'' diffusion coefficient or diffusivity.  The Green's function of Eq.~(\ref{e:diffeq}) is a Gaussian
\begin{equation}
\rho(\tilde{\bf x},t)=\frac{e^{-\frac{D^{-1}_{ij}\tilde x_i\tilde x_j}{4t}}}
{\sqrt{4\pi t\det {\bf D}}}
\end{equation} 
involving the inverse matrix ${\bf D}^{-1}$. The second moment
of position with respect to this density is
\begin{equation}
\int\rho(\tilde{\bf x},t)\tilde x_i\tilde x_jd{\bf x}=2D_{ij}t
\end{equation}
The macroscopic picture may be connected to the
microscopic dynamics in terms of the displacement of a single particle $\Delta(t)=\tilde{\bf x}(t)-
\tilde{\bf x}(0)$.  If we assume the microscopic dynamics is periodic with respect to a lattice of
finite covolume, we can define a configuration variable $\bf x$ modulo lattice translations, and
consider the dynamics as a random process defined using a probability measure $\mu$ on its initial
position and velocity $({\bf x}(0),{\bf v}(0))$.  Expectations with respect to
this measure will be denoted by angle brackets $\langle\rangle$.  Depending on the details of the
dynamics, it may be possible to prove that in the $t\to\infty$ limit $\Delta$ has the appropriate second
moment
\begin{equation}
\lim_{t\to\infty}\frac{\langle \Delta_i\Delta_j\rangle}
{2t}=D_{ij}
\end{equation}
and/or that it converges weakly to a normal distribution
\begin{equation}
\frac{\Delta}{t}\Rightarrow {\cal N}(0,\Sigma_{ij})
\end{equation}
with zero mean and covariance matrix $\Sigma_{ij}=2D_{ij}$.
Convergence of the process to Brownian motion (ie multiple
time distributions) with the same covariance matrix has
also been shown in some cases.  Note in particular
the factors of 2.  We may want an overall diffusivity, generalising the isotropic case:
\begin{equation}
D=\frac{1}{d}{\rm tr}{\bf D}=\lim_{t\to\infty}\frac{1}{2dt}
\langle\Delta_i\Delta_i\rangle  
\end{equation}
for example Bleher~\cite{B} uses this for $d=2$.

The dynamics may lead to anomalous diffusion, where
the second moment may
not be linear in $t$.  Conventions differ as to how
the diffusivity is generalised.  Here we consider cases
where the second moment is logarithmic.  By analogy with
the normal distribution case above we define
\begin{equation}\label{e:super}
\lim_{t\to\infty}\frac{\langle\Delta_i\Delta_j\rangle}
{2t\ln t}={\cal D}_{ij},\qquad
{\cal D}=\frac{1}{d}{\rm tr}{\cal D}_{ij},\qquad
\frac{\Delta}{t\ln t}\Rightarrow {\cal N}(0,\Xi_{ij})
\end{equation}
where the relevant limits exist, but now it is no longer clear
whether ${\cal D}_{ij}$ is related to $\Xi_{ij}$, as noted
in~\cite{CESZ}.  Convergence in distribution does not imply
convergence of the moments, since the tails may decay slowly
while spreading further with time; see Fig.~\ref{f:dist}.
The authors of Ref.~\cite{CESZ} numerically
considered higher moments of $\Delta(n)/\sqrt{n\ln n}$ but
it is likely that recent methods~\cite{MT} will lead to a proof
of an anomalous exponent without a logarithm for these
higher moments and with a logarithm for lower
moments~\cite{Mpc}.  This proposed work does not include
our case of interest (second moment) which is borderline
between regimes of the normal and abnormal moments.
We continue the discussion of this anomalous moment below.

\begin{figure}
\includegraphics[width=400pt, height=250pt]{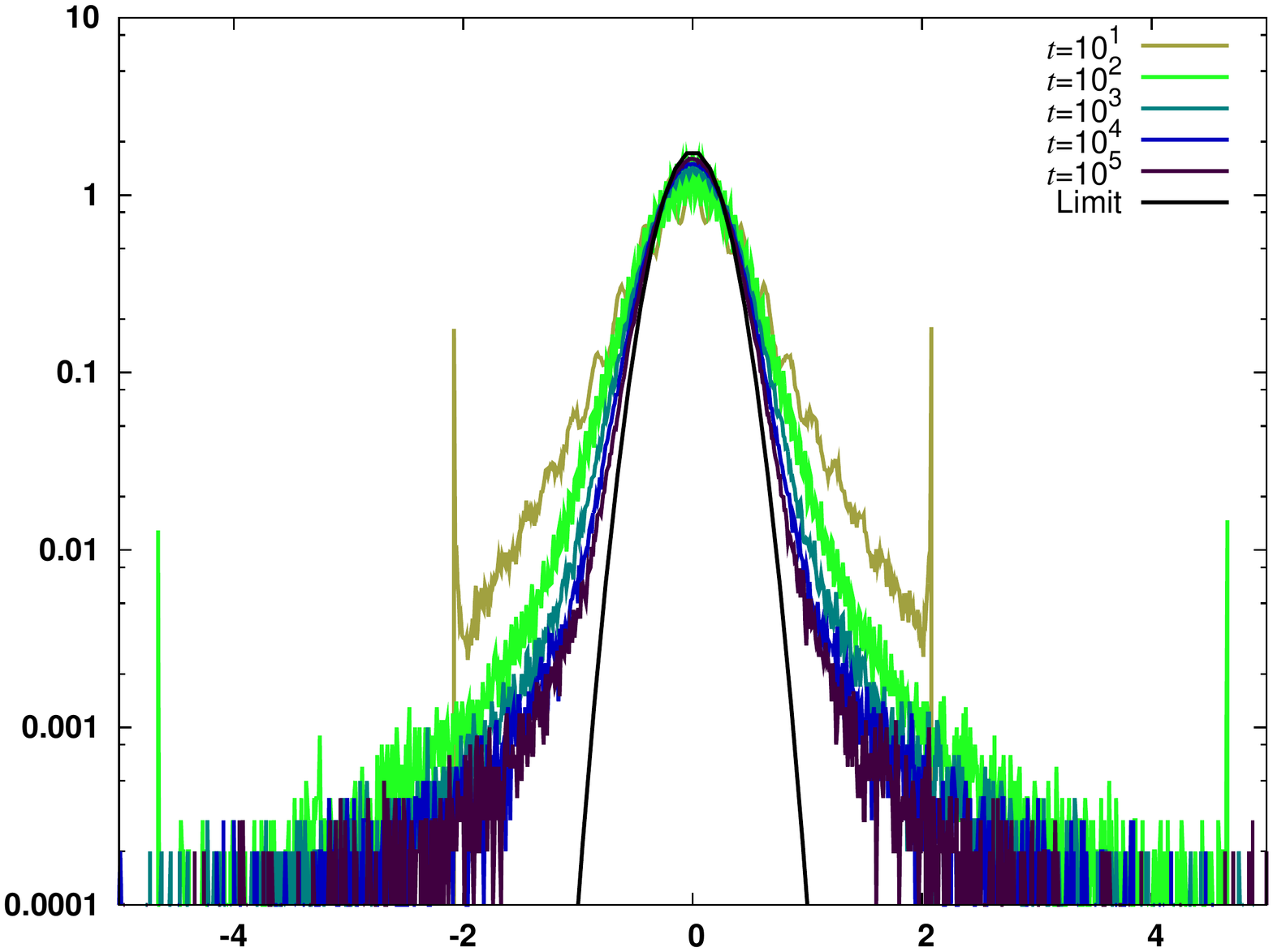}
\caption{\label{f:dist} The distribution of the scaled displacement $\Delta/\sqrt{t\ln t}$ for the
square ($d=2$) Lorentz gas with $r=0.4$, showing convergence to the expected normal
distribution, but with fat expanding tails at any finite time that support the argument in the
text for an anomolously large second moment.  Two features observed here at short times are
oscillations due to the lattice, and spikes arising from ballistic orbits with no collisions
that indicate the boundaries of the support of the distribution.}
\end{figure}

As with the normal
diffusion, we assume that superdiffusion is possible
in all directions; in general it is possible to have
normal diffusion in some directions and superdiffusion
in others, as discussed in~\cite{SV,CD}.  Also, the
definition for $\Xi$ makes sense with respect to discrete
(collision) time $n$ replacing $t$; we expect the relation
$\Xi^{\rm disc}=\tau\Xi$ with $\tau$ the
mean free time; see~\cite{B,CD} and Sec.~\ref{s:disc}.
Note however that ${\cal D}^{disc}$ is problematic since
$\langle\Delta^2\rangle$ is infinite for a single free path as well
as $n$ free paths, and hence not susceptible to normalisation
by any function of $n$.

Noting $\Delta_i=\int_0^t v_i(s)ds$ we find
(see Eq.~1.4 in Ref.~\cite{B}) the expression
\begin{equation}\label{e:GK}
\frac{1}{2}\langle\Delta_i\Delta_j\rangle
=t\int_0^t\langle v_i(0)v_j(s) \rangle ds
-\int_0^ts\langle v_i(0)v_j(s) \rangle ds
\end{equation}
which leads to the usual continuous time Green-Kubo
formula for $D_{ij}$ in the limit $t\to\infty$,
assuming the velocity autocorrelation function (VACF)
decays sufficiently rapidly.  If instead, we have
$\langle v_i(0)v_j(s)\rangle\sim C_{ij}/s$, we find
${\cal D}_{ij}=C_{ij}$.  See Ref.~\cite{M} for
rigorous results on continuous time decay of
correlations for this system.
Following Bleher~\cite{B}, we
note that in this (continuous time) case, the superdiffusion
arises from the slow decay of the VACF, but that in the
discrete case, the correlations decay rapidly, but
superdiffusion arises from the infinite
second moment of the displacement between collisions.

We now return to the specific case of the periodic
Lorentz gas and discuss existing results for $d=2$.
The Lorentz gas is a billiard in an unbounded domain with
a periodic collection of convex obstacles. If the free
path length is bounded away from zero and infinity (``finite
horizon''), the
diffusion is normal, as noted in the Introduction.  The
Green-Kubo formula is valid, but there is no closed form
expression for the correlation functions or diffusion coefficient.

In the infinite horizon case there are infinite strips
that do not contain any scatterers, termed ``corridors''
that lead to logarithmic superdiffusion as described
above.  These corridors can be enumerated, leading to
exact expressions for the superdiffusion coefficient.
We assume that there are at least two non-parallel
corridors, so that superdiffusion occurs in all
directions.

For a comparison of the coefficients existing in the literature,
we need only consider the specific case of a Lorentz gas with a square
lattice of unit spacing and disks of radius $1/\sqrt{8}<r<1/2$, for
example the case $r=0.4$ considered in numerical simulations below.
This has only corridors in the horizontal and vertical directions. 
Bleher~\cite{B} following Eq.~8.5 writes a discrete
superdiffusion coefficient (in our notation)
\begin{equation}
{\cal D}^{\rm disc}_{\rm Bleher}=\frac{w^2}{2\pi r}
\end{equation}
where $w=1-2r$ is the width of the corridor, and $2\pi r$ is
the circumference of a scatterer.  As noted above the second moment of
$\Delta$ is actually infinite.  Bleher attempts to circumvent this issue by
putting the $n\to\infty$ limit inside the average, but this also does
not make sense at face value.  If we make the interpretation
${\cal D}^{\rm disc}=\tau{\cal D}$ it agrees with the calculations presented here.

Szasz and Varju (\cite{SV}, Eq. (2)), and later Chernov and Dolgopyat (\cite{CD},
Eq. (2.1)), give for the components of the covariance matrix
\begin{equation}
{\Xi}^{\rm disc}_{ij}=\sum_{x}\frac{c_\mu w_x^2}{2|\psi(x)|}
\psi_i(x)\psi_j(x)
\end{equation}
where $x$ is a fixed point of the collision map of the reduced (unit cell) space,
assuming that each side of the corridor contains periodic copies of only one
scatterer; there are four of these for each corridor.
$c_\mu=1/(4\pi r)$ is the normalising factor for the invariant measure of this map,
equal to half the reciprocal perimeter of the billiard.
$w_x$ is the width of the corridor corresponding to $x$.
Finally $\psi(x)$ is the vector parallel to the corridor
giving the translation of $x$ under the map in the full
space, with components $\psi_i(x)$.
For the square lattice, $\psi(x)$ is a unit vector
aligned with the coordinate axes, and taking account of
the above mentioned factor of four, we find
\begin{equation}\label{e:Xi=D}
{\Xi}^{\rm disc}_{ij}=\frac{w^2}{2\pi r}\delta_{ij}
\end{equation}
Thus in contrast to the normal diffusion case it appears
that $\Xi\neq 2{\cal D}$, in fact $\Xi={\cal D}$.
This discrepancy is not noted in existing literature,
presumably due to the different notations used.  However,
Balint, Chernov and Dolgopyat have reportedly proposed it
for this (two dimensional) case, independently of and
prior to the present author, and claim to have a proof~\cite{C11}.

From Eq.~(\ref{e:GK}) we can (heuristically) explain the discrepancy as follows.  Given a random
initial condition (according to the usual continuous time invariant probability measure $\mu$),
there is a probability $\sim C/t$ that it will make a free flight for at least a long time $t$.
This is much higher than the exponentially small probability of travelling this distance
according to the normal distribution.  Thus the latter
can be accurate at best only to distances
\begin{equation}
\frac{e^{-x^2/(2\Xi t\ln t)}}{\sqrt{2\pi\Xi t\ln t}}\approx
\frac{C}{x}
\end{equation}
that is,
\begin{equation}\label{e:crossover}
x\approx \sqrt{\Xi t\ln^2 t}
\end{equation}
Orbits dominated by a free flight much longer than this are
well outside the converged part of the normal distribution
and thus cannot contribute to its variance.  Splitting the
integrals in Eq.~(\ref{e:GK}) into times less than, or greater
than $\sqrt{t}$, we find that the first term gives a
contribution to $\cal D$ from the orbits that
could contribute to the normal distribution of $(C/2)t\ln t$,
that is, exactly half its full value (modulo lower order
terms like $t\ln\ln t$ arising from the logarithm in Eq.~(\ref{e:crossover})).
Thus we expect $\Xi={\cal D}$ in agreement with Eq.~(\ref{e:Xi=D}).
Note that this argument depends on the specific
mechanism for the slow decay of the VACF, and may not
apply to all processes with logarithmic superdiffusion.  It should
however extend to the Lorentz gas in arbitrary dimension considered
for the remainder of this paper.

Thus we propose that the second moment is anomalous, specifically by a factor of two,
giving the equation $\Xi={\cal D}$ in two and higher dimensions. In $d=2$ we agree
with recent rigorous expression for $\Xi$ in Refs.~\cite{CD,SV} and Bleher's $\cal D$.
However, some of Bleher's discussion of ${\cal D}^{\rm disc}$ and $\Xi$ appears problematic.  

\section{Summary of results}\label{s:summary}

As discussed previously for the case $d=2$~\cite{B,CD,SV} superdiffusion in an infinite horizon Lorentz gas is a result of the unbounded path length of the particle, thus an understanding requires close study of the properties of these paths.  Here we extend the concept of a horizon (or ``channel'') as it appears in existing literature. Usual terminology defines free paths as bounded from
above (``finite horizon'') or infinite (``infinite horizon'') with occasional variations (``finite but unbounded'', also called ``locally finite'', for some aperiodic Lorentz gases).
Sanders~\cite{S} makes a finer classification
of infinite horizon billiards in three dimensions using the dimensionality of the horizon (cylindrical or planar horizon billiards).  Here we define a horizon, not as a type of bound or dimension, but as a set of points whose trajectories have no collisions with the scatterers, and satisfy some natural conditions.

Some definitions we need in this section (with more details in Sec.~\ref{s:def}) are as follows.
We consider the compact configuration space $M_x\subset\mathbb{R}^d/{\cal L}$ where the lattice
$\cal L$ of covolume $\cal V$ is the discrete subgroup of $\mathbb{R}^d$ defining the periodicity
of the Lorentz gas.  The scatterers cover a volume fraction ${\cal P}=1-|M_x|/{\cal V}$.
The continuous time dynamics $\Phi^t$ acts on the phase space $M=M_x\times M_v$ with
$M_v=\mathbb{S}^{d-1}$, the set of velocities of unit magnitude.  The set $M_t\subset M$
denotes initial conditions that do not collide with a scatterer for times in $[0,t]$.
A horizon $H\subset M_\infty$ is of the form $H=H_x\times H_v$ with $H_v=\mathbb{S}^{d-1}\cap V_H$
and $V_H\subset \mathbb{R}^d$ is a linear space of dimension $d_H\in\{1\ldots d-1\}$, the ``dimension
of the horizon.''  The set $H_x$ is connected and invariant under translations in $V_H$.
Projecting onto the $d-d_H$ dimensional orthogonal space $V_H^\perp$ we obtain a lattice
${\cal L}_H^\perp$ of covolume ${\cal V}_H^\perp$ and a set characterising the horizon $H_x^\perp$.
There is a natural invariant probability measure on $M$, denoted $\mu$.  The probability of
surviving without a collision for time $t$ is $F(t)=\mu(M_t)$, while the probability of remaining
within a horizon $H$ (or more precisely its spatial projection) is
\begin{equation}\label{e:F_H}
F_H(t)=\mu(\{({\bf x},{\bf v}):\Phi^s({\bf x},{\bf v})\in H_x\times M_v,\quad \forall s\in[0,t]\}).
\end{equation}

We make further distinctions, illustrated in Fig.~\ref{f:hor}:  A maximal horizon is one of highest dimension for the given billiard, a principal horizon is one of the highest dimension possible, $d-1$ (planar in Sanders' case) and an incipient horizon is one where $H_x^\perp$ has zero $d-d_H$ dimensional Lebesgue measure, hence $F_H(t)=0$ for all $t$.
For example, a principal horizon has a connected
one-dimensional set $H_x^\perp$ that is an interval, of length $w_H$, the ``width'' of the horizon.
When this interval shrinks to a point, as by increasing the size of the scatterers, the horizon
has scatterers touching it on both sides and becomes incipient. The set of maximal non-incipient
horizons is denoted $\cal H$.  Lemma~\ref{l:fin} in Sec.~\ref{s:lem} asserts that this set is finite. 

\begin{figure}
\centerline{
\includegraphics[width=175pt]{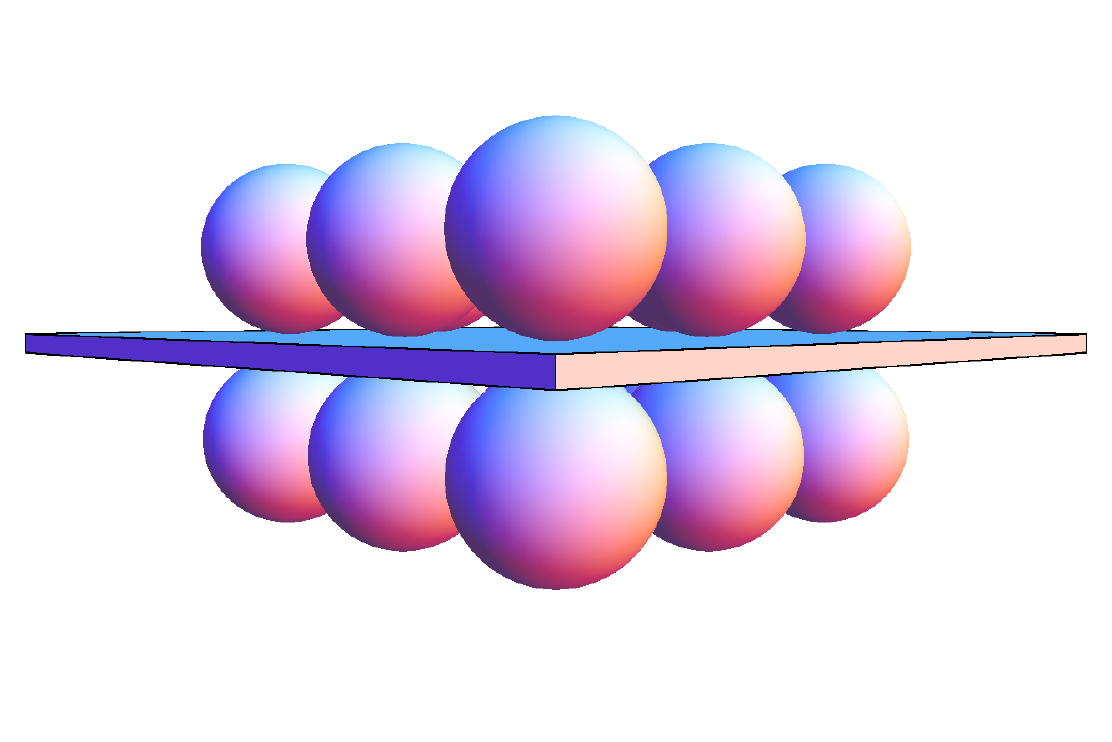}
\includegraphics[width=175pt]{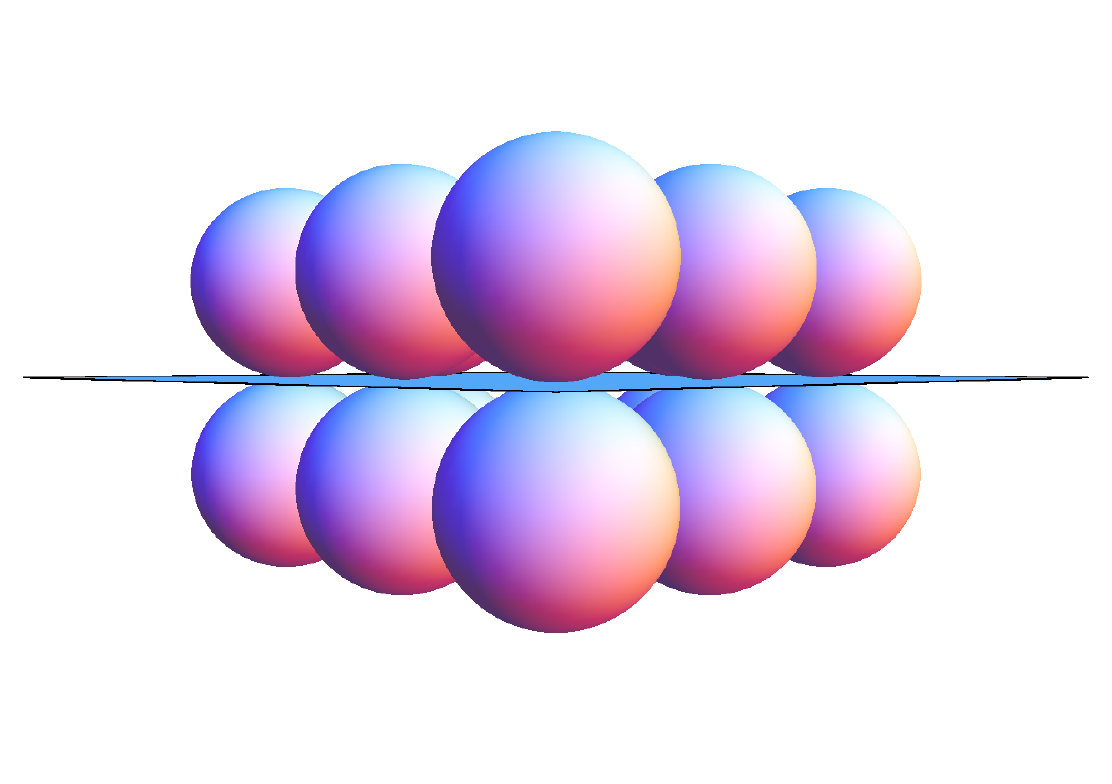}
\includegraphics[width=120pt]{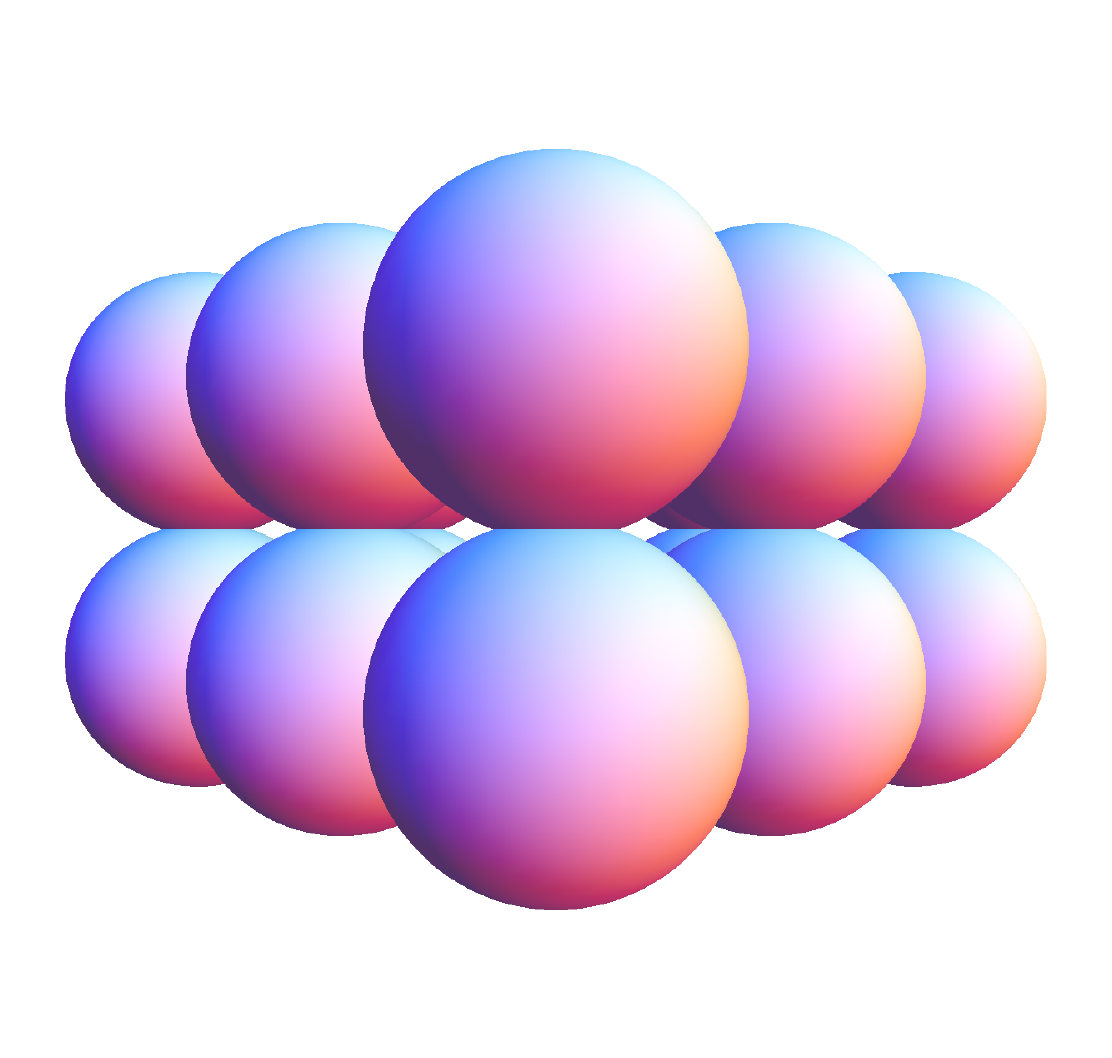}}
\caption{\label{f:hor} Horizons in the cubic Lorentz gas, depicted with a $3\times 3\times 2$
subset of the lattice, as viewed from
one corner, for different scatterer radii, and illustrating the
horizon terminology introduced in the text.  If the scatterers are small enough that they don't touch ($r<0.5$, left) there are planar horizons that are principal (and hence maximal) and not incipient.  When the scatterers touch ($r=0.5$, centre), they create an incipient principal horizon.  When the spheres overlap ($r>0.5$, right), cylindrical horizons are possible as seen in the figure; these are maximal but
not principal.}
\end{figure}

We will see that the decay exponent of $F(t)$ is determined by the dimensionality of the maximal
horizon with special conditions related to incipient horizons, and ${\cal D}_{ij}$ and ${\cal D}_{ij}^{\rm disc}$ are nonzero only for non-incipient principal horizons.  

Although much progress has been made on higher dimensional billiards in the last decade, some major challenges remain~\cite{Sz}.  Existing results are not sufficient to provide a fully unconditional foundation for the study of diffusion in higher dimensional Lorentz gases, thus our results depend
on the following conjectures.

First, we assert that the leading order collision-free behaviour is determined only by maximal horizons, except where these are all incipient.  This means in particular that incipient horizons, non-maximal horizons and overlaps between maximal horizons are all assumed to give subleading contributions.

\begin{conj}\label{c:hor}
Consider a periodic Lorentz gas with at least one
non-incipient maximal horizon.  Then we have
\begin{equation}
\sum_{H\in{\cal H}}F_H(t)\sim F(t)
\end{equation}
as $t\to\infty$.
\end{conj}
The initial sentence of the conjecture implies that $F(t)>0$ for all time as required by the right hand side of the equation.  The definition of periodic Lorentz gas considered in this paper is given
at the start of Sec.~\ref{s:def}.

\begin{remark} This conjecture describes only the horizons and not reflections from the scatterers.  It may not be necessary for the scatterers to be dispersing, as its content is purely geometrical, and
not dynamical.  The main obstruction to proving it is probably the effect of maximal incipient
horizons (see the discussion preceding Conj.~\ref{c:inc} below).
\end{remark}

\begin{remark} Where there are only incipient horizons, it is likely that $F(t)$ must decay to zero
at finite time.  We do not discuss this case here.
\end{remark}

Now we are ready to state the first main result:

\begin{table}
\centerline{
\begin{tabular}{ccccccc}\hline\hline
$d$&0&1&2&3&4&general\\\hline
$S_d$&2&$2\pi$&$4\pi$&$2\pi^2$&$\frac{8\pi^2}{3}$&$\frac{2\pi^{(d+1)/2}}{\Gamma((d+1)/2)}$\\
$V_d$&1&2&$\pi$&$\frac{4\pi}{3}$&$\frac{\pi^2}{2}$&
$\frac{\pi^{d/2}}{\Gamma((d+2)/2)}$\\
$G_d$&-&-&$\frac{1}{2\pi}$&$\frac{1}{4}$&$\frac{1}{\pi}$&
$\frac{\Gamma(d/2)}{2\sqrt{\pi}\Gamma((d-1)/2)}$     \\\hline\hline
\end{tabular}}
\caption{Lebesgue measure of $d$-dimensional spheres, $S_d$ and 
balls $V_d$, and the combination $G_d=S_{d-2}/2S_{d-1}$
appearing in Eq.~\protect\ref{e:sum} below, including specific values of interest
and the general formulae (last column).\label{t:Sd}}
\end{table}

\begin{thm}\label{th:F}
Consider a periodic Lorentz gas satisfying Conjecture~\ref{c:hor}.
Then the free flight function $F(t)$ satisfies
\begin{equation}
F(t)\sim\sum_{H\in\cal H}\frac{S_{d_H-1}
\int_{H_x^\perp}\int_{H_x^\perp}
\Delta^{\rm vis}({\bf x},{\bf y})d{\bf x}d{\bf y}}
{S_{d-1}{\cal V}_H^\perp(1-{\cal P})t^{d-d_H}}
\end{equation}
as $t\to\infty$.
\end{thm}
Here, $\Delta^{\rm vis}({\bf x},{\bf y})$ is the ``visibility''
counting function, giving the (non-negative integer) number of ways $\bf x$ and $\bf y$
can be connected by a straight line entirely in $H_x^\perp$; for the specific examples
considered here, this is either zero or one.  $S_d$ is defined in Tab.~\ref{t:Sd}.

\proof{Conjecture~\ref{c:hor} asserts that only horizons in $\cal H$
need be considered to determine the asymptotic form of $F(t)$.  Lemma~\ref{l:fin} (below) shows that the sum over $\cal H$ is finite, and lemma~\ref{l:horsum} (below) gives the value of each term in the sum.
\hfill\qed}

Sec.~\ref{s:princ} gives the main application of this theorem, the expression~(\ref{e:sum}) for $F(t)$ in the case of the cubic Lorentz gas
with principal horizons.

If all maximal horizons are incipient,
we cannot apply Conj.~\ref{c:hor}.
However heuristic arguments and numerical simulations given in Sec.~\ref{s:inc} suggest in the case of a principal incipient
horizon, contributions from the infinite number of horizons of next lower dimension can be
summed, giving a $t^{-2}$ decay of $F(t)$, as long as $d<6$.  For $d\geq 6$ the sum of
contributions diverges, and it appears that $F(t)$ has a decay intermediate between $t^{-2}$ and $t^{-1}$.  Assuming a logarithm at the critical dimension
$d=6$ we can rather tentatively suggest:

\begin{conj}\label{c:inc}
The free flight survival probability for a Lorentz gas with incipient
(but no actual) principal horizon satisfies
\begin{equation}
F(t)\asymp\left\{\begin{array}{cr}
t^{-2}&d<6\\
t^{-2}\ln t&d=6\\
t^{-\alpha_d}&1<\alpha_d<2, \qquad d>6
\end{array}\right.
\end{equation}
\end{conj} 

Finally we turn to the (super-)diffusion coefficient.  Here, a
further conjecture is required to establish estimates on velocity autocorrelation functions. Consider a flight of length $t$ much larger than the lattice spacing; this has an
angle from the horizon of order $t^{-1}$.  The particle
makes a grazing collision with a scatterer and reflects with a typical angle of order $t^{-1/2}$ so that the subsequent flight is typically of order $\sqrt{t}$ and further flights shorter still.  Thus after any time proportional to $t$,
we expect that many collisions have completely randomised
the trajectories position relative to the scatterer and velocity direction, but not position in the full lattice.
We state this as

\begin{conj}\label{c:corr}
Consider a periodic Lorentz gas with flow $\Phi^t$ and free flight function
$F(t)>0$ for all $t>0$.
Let $f,g:M\to\mathbb{R}$ denote zero-mean (with respect to
$\mu$) H\"older functions. Then we have
\[ \int_{M\backslash M_t}
(f)(g\circ\Phi^t)d\mu=o(F(t)) \]
as $t\to\infty$.
\end{conj}

Note that the $\sqrt{t}$ motivating the above conjecture holds for the Lorentz gas, but not for other models such as extended stadium billiards. Also, a connection between
free flights and correlations for the three dimensional
cubic Lorentz gas was proposed as early as 1984~\cite{FM}.
We can now address the diffusion properties of the Lorentz gas:

\begin{thm}\label{th:D}
Consider a periodic Lorentz gas with non-incipient principal horizon.  Assume
Conjectures~\ref{c:hor} and ~\ref{c:corr}.  Then the super-diffusion matrix
is given by
\begin{equation}
{\cal D}_{ij}
=\frac{1}{1-{\cal P}}\frac{V_{d-1}}{S_{d-1}}
\sum_{H\in{\cal H}}\frac{w_H^2(\delta_{ij}-n_i(H)n_j(H))}{{\cal V}_H^\perp} 
\end{equation}
where ${\bf n}(H)$ is either of the two unit vectors normal to
all velocities in $H$.
\end{thm}
Note that in contrast to two dimensions, it is preferable to describe
horizons by their perpendicular rather than parallel directions, hence
the projection operator appearing in this formula.

Combining both theorems and taking the trace of the diffusion matrix,
we find the simple relation:
\begin{equation}\label{e:DF}
d{\cal D}=\lim_{t\to\infty}tF(t)
\end{equation}
independent of the details of the principal horizons.

In the case that the principal horizons do not span the full space,
the superdiffusion matrix is degenerate, and we expect that as in
two dimensions, there will be normal diffusion in the direction(s)
orthogonal to all these horizons if such motion is not completely blocked.
Sec.~\ref{s:diff} contains a numerical test of Conj.~\ref{c:corr} and the
proof of Thm.~\ref{th:D}, with application to the discrete time diffusion
coefficient given in Sec.~\ref{s:disc}.

For both the free flight function and the diffusion coefficient, the limit
of small $r\to 0$ spherical scatterers leads to an increasing number
of horizons.  Taking a Mellin transform, both results can be formally related
to the Riemann hypothesis (RH), the celebrated conjecture that the complex
zeros of the Riemann zeta function $\zeta(s)$ lie on the line with real part
$1/2$~\cite{Co}.  The formal argument assumes that the contour integral can be
closed avoiding the poles of $\zeta(s)^{-1}$ in the integrand in a
convergent sequence of increasing semicircles.  A complete proof of an equivalence with RH is not yet available even for the simpler related case
of the open circular billiard of Ref.~\cite{BD}; work on the latter is
ongoing.  Connections between lattice sums and RH is of course not new; see
for example~\cite{Ba}. Here we make a connection between RH and transport in
a definite physically motivated model; other physical systems with connections
to RH are discussed in Ref.~\cite{SH}.

Further development of this subject could include verification and/or
modification of the above conjectures on the theoretical side; the incipient
case is particularly challenging due to the similarity with the example in
Ref.~\cite{BT}, hence almost certainly exponential growth of complexity.
The case of multiple rationally located scatterers per unit cell naturally
leads to Dirichlet twisted Epstein zeta functions, which may be interesting
from a number theoretical point of view (see the discussion at the end of
Sec.~\ref{s:princ}).  There are generalisations of the Lorentz gas which are
almost periodic~\cite{DSV,NSV} or have external fields~\cite{CD,CD2,D}.

On the applications side, periodic boundary conditions are
commonly used in molecular dynamics simulations, so a natural extension would
be to investigate the (possibly unphysical) effects of these boundary conditions
in simulations of low density systems of many particles.  Transport in periodic
structures with long free flights occur naturally in crystalline materials, nanotubes
and periodically fabricated devices such as antidot arrays and various classes
of metamaterials.  For these, it should be noted that the soft potentials may
modify the dynamics of particles moving close to a horizon direction.  As with
periodic time perturbations, the effect could be either more or less stable
than free motion in the Lorentz gas depending on the details of the potential
and the speed of the particle, however, as here, enumeration of horizons
remains important for understanding transport properties.

The remaining sections are structured as follows: Sec.~\ref{s:def} motivates
and gives the necessary definitions related to the horizons,
Sec.~\ref{s:lem} develops the proof of Thm.~\ref{th:F},
proving two lemmas and discussing Conj.~\ref{c:hor}.  Sec.~\ref{s:princ}
then applies this theorem to principal horizons, giving $F(t)\sim C/t$ with the 
constant determined exactly.  Numerical methods are then discussed, and
used to confirm these results. In
the limit of small scatterers, connection is made with the Riemann
Hypothesis.  Sec.~\ref{s:inc} then considers incipient horizons, where heuristic
arguments and numerical simulations suggest Conj.~\ref{c:inc}.  Sec.~\ref{s:diff} then turns to the diffusion problem, proving Thm.~\ref{th:D}
conditional on Conjs.~\ref{c:hor} and~\ref{c:corr},
again for principal horizons.  In this case the numerical results
are inconclusive due to the logarithm being dominated by a constant
at accessible numerical timescales.
Finally, Sec.~\ref{s:disc} treats the diffusion problem in discrete time.

\section{Definitions}\label{s:def}
The phase space of the full periodic Lorentz gas is $\tilde M=\tilde M_x\times M_v$ where
$\tilde M_x\subset\mathbb{R}^d$ denotes the configuration space external to the scatterers and $M_v=\mathbb{S}^{d-1}$ is the set of velocities of unit magnitude.  The same quantities without tildes are associated with the compact phase space of the dynamics considered modulo the lattice, so $M=M_x\times M_v$ with $M_x=\mathbb{R}^d/{\cal L}\backslash\Omega$.  There are natural projections
$\pi_x:M\to M_x$ and $\pi_v:M\to M_v.$  Here,
$\cal L$ is the lattice of periodicity of the Lorentz gas, ie a discrete subgroup of $\mathbb{R}^d$ with finite covolume $\cal V$,
and in the main example given by $\mathbb{Z}^d$.  Also, $\Omega\subset\mathbb{R}^d/{\cal L}$
denotes the union of the scatterers.
The scatterers (connected components of $\Omega$) are assumed to be finite in number.  The curvature
of their boundaries has a positive lower bound, and the boundaries themselves are $C^3$ smooth.
To allow for touching or overlapping spheres (for example), the scatterer boundaries may be non-smooth
on a finite number of submanifolds of dimension at most $d-2$; we require however that each scatterer
be the union of a finite number of convex sets in this case. 

The uniform probability measure on $M$ is $\mu=m/[{\cal V}S_{d-1}(1-{\cal P})]$ where $m$ is the
usual Lebesgue measure for $\mathbb{R}^d\times\mathbb{S}^{d-1}$, $S_{d-1}$ is the Lebesgue measure of the sphere $\mathbb{S}^{d-1}$ (see Tab.~\ref{t:Sd}) and $\cal P$ is the packing fraction of the obstacles,
equal to $m_x(\Omega)/{\cal V}$ where $m_x$ is Lebesgue measure on $\mathbb{R}^d$.

The usual billiard flow is denoted $\Phi^t:M\to M$.  Thus we have the free flight formula
$\Phi^t({\bf x},{\bf v})=({\bf x}+{\bf v}t\;\;{\rm mod}\;{\cal L},{\bf v})$
in the absence of a collision.  The time to the next collision is denoted
$T({\bf x,v})=\inf\{t>0:\pi_x\circ\Phi^t({\bf x},{\bf v})\in \partial\Omega\}$.
If there is no future collision, we allow $T({\bf x,v})$ to be infinite.
The set surviving for time $t$ without a collision is
$M_t=\{({\bf x,v})\in M:T({\bf x,v})\geq t\}$, in terms of which $F(t)=\mu(M_t).$
One of our main aims is to characterise the large $t$ asymptotics of this function.

The
flow $\Phi^t$ is discontinuous at a collision, using the usual mirror
reflection law, but we do not need the explicit formula here.  Nor do we need
to specify conventions for the flow at the exact moment of collision or for
orbits that are tangent to the boundary or reach a non-smooth point on it; these
sets are of zero measure.

In three or more dimensions, collections of orbits with large or infinite
$T$ are of different types.  We now formalise the terms introduced in the
previous section.  We want to characterise subsets of $M_\infty$, ie
collections of points such that $T=\infty$.  

The free dynamics is uniform motion on a torus.
Thus a trajectory starting from a point in $M$
must either be dense (hence have finite $T$)
or move parallel to a linear subspace spanned by vectors in
${\cal L}$.  Thus for each ${\bf x}\in\pi_xM_\infty$ there are one or
more (possibly infinitely many) linear subspaces of directions in
which there are no collisions. Let us choose one, $V_H$ which is not
contained in another such subspace.  Then, we can define
\begin{eqnarray}
H_v&=&V_H\cap \mathbb{S}^{d-1}\label{e:Hv}\\
\hat{H}_x&=&\{{\bf x}\in M_x:({\bf x},H_v)\in M_\infty\}
\end{eqnarray}
The equation defining free motion ensures that
$\hat{H}_x$ is invariant under translations in $V_H$.  Thus we
can decompose it as
\begin{equation}\label{e:Hxperp}
\hat{H}_x=\hat{H}_x^\perp\oplus V_H/{\cal L}_H
\end{equation}  
where ${\cal L}_H={\cal L}\cap V_H$ is a sublattice of the original
lattice $\cal L.$  Initial conditions contained in $\hat{H}_x$ and having
velocities close to vectors in $H_v$ will survive for a long time
before colliding, as long as they remain in a connected component
of $\hat{H}_x$.  This motivates the following definition of a horizon:
\begin{defn}
A Horizon is a set $H\subset M_\infty$ which satisfies
\begin{description}
\item[(a)] $H$ is a direct product of sets $H_x\subset M_x$ and
$H_v\subset M_v$,
\item[(b)] $H$ is inextensible, ie not a proper subset of another
horizon,
\item[(c)] $H_x$ is connected.
\end{description}
\end{defn}
The above discussion implies that all horizons are of the
form given in Eqs.~(\ref{e:Hv}--\ref{e:Hxperp}) above, with the
additional condition that $H_x$ be a connected component of
$\hat{H}_x$, that is,
\begin{eqnarray}
H&=&H_x\times H_v\\
H_x&=&H_x^\perp\oplus V_H/{\cal L}_H\\
H_v&=&V_H\cap\mathbb{S}^{d-1}
\end{eqnarray}
The dimension
of the horizon $d_H\in\{1\ldots d-1\}$
is defined to be the dimension of $V_H$, which
is also the dimension of ${\cal L}_H$.
When $d_H=1$, $H_v$ consists of two opposite points, so $H$ itself
is not connected, even though $H_x$ is connected.  When $d_H>1$,
$H$ is also connected.

An example where $\hat{H}_x$ may have
more than one connected component is that of a square
lattice, ie ${\cal L}=\mathbb{Z}^2$, with 
a circular scatterer of radius $r_1=0.2$ at the origin and
an additional scatterer of radius $r_2=0.1$ at $(1/2,1/2)$.
Particles moving horizontally will not hit any scatterer if their
$y$ coordinate lies in either of the intervals $(0.2,0.4)$, 
$(0.6,0.8).$  Thus there are two horizons with the same $V_H$
in this case.

We denote the $d-d_H$ dimensional covolume of ${\cal L}_H$ in $V_H$
by ${\cal V}_H$.
The space of dimension $d-d_H$ perpendicular to
$V_H$ is denoted $V_H^\perp$.  The original lattice induces a
lattice ${\cal L}_H^\perp\subset V_H^\perp$, given by elements
of ${\cal L}/{\cal L}_H$ projected onto $V_H^\perp$ along $V_H$.
This has covolume ${\cal V}_H^\perp={\cal V}/{\cal V}_H$ since a
unit cell of the original lattice
can be formed by the product of two unit cells of lower
dimension in orthogonal subspaces.
In the original space $\tilde{M}$
each horizon $H$ corresponds to a periodic collection
of connected components;
lattice translations in ${\cal L}_H$ preserve each component, while lattice
translations in $\cal L\backslash{\cal L}_H$ map one component to another.

When $d=2$ we can have only horizons with dimension $d_H=1$,
corresponding to the channels in previous literature
(see Sec.~\ref{s:DXi}).
When $d=3$ we can have a ``planar horizon'', $d_H=2$, in
which $H_v$ is a great circle in $\mathbb{S}^2$, a
``cylindrical horizon'', $d_H=1$, in which $H_v$ consist of two
symmetrically placed points; this terminology comes from~\cite{S}.
More than one type of horizon may exist in the same billiard. For
example if the scatterers do not cross the boundaries of the
unit cell in three dimensions, there are planar horizons along
the boundaries; there may also be a horizon of cylindrical type in the interior.
Finite horizon, in which no horizons exist, are also possible in
any dimension $d$, and becomes increasingly probable for random
arrangements of scatterers of fixed density in a unit cell of
increasing size.  Nevertheless no explicit example of a finite horizon
periodic Lorentz gas in $d>2$ has yet been shown.

Recall the discussion of the free flight function, $F(t)$
above.  We can similarly define the probability of remaining for
at least a given time in the spatial projection of a horizon,
$F_H(t)$; this definition has already been given in Eq.~(\ref{e:F_H}).
It is useful to make some definitions related to the dimension:

\begin{defn}
Maximal horizon: A horizon of dimension greater than or equal to
all other horizons of the given billiard.
\end{defn}

\begin{defn}
Principal horizon: A horizon $H$ of the highest possible dimension,
that is, $d_H=d-1$.
\end{defn}

Clearly a principal horizon is also maximal.  Anomalous diffusion
properties arising from a $1/t$ decay of velocity autocorrelation
functions are expected only for principal horizons;
see Ref.~\cite{M,S} and below.

We now make an important observation:
For dimensions $d>2$ we have a qualitatively
new type of horizon, corresponding to the border
between horizons of one dimension and
the next.

\begin{defn}
Incipient horizon: A horizon for which $H_x^\perp$ has zero $d-d_H$
dimensional Lebesgue measure.  An incipient maximal or incipient
principal horizon will denote a
horizon that is incipient and also maximal or principal, respectively.
\end{defn}

When $d=2$, an incipient horizon is an orbit tangent to scatterers on
both sides, hence isolated and having no effect on the dynamics;
this is a rather trivial form of horizon.
In dimensions $d>2$, a incipient horizon occurs when
$H_x^\perp$ for a horizon of dimension $d_H>1$ shrinks to a zero
measure set, so that $F_H(t)=0$.  This is easiest to visualise in the
three dimensional cubic Lorentz gas,
when the planar horizons shrink to zero width touching scatterers on
both sides as the radius approaches $1/2$ from below; see Fig.~\ref{f:hor}.

Within an incipient horizon $H$, there is a lattice of points
${\cal L}_{inc}$ in $H_x$
where the scatterers touch the horizon. The set of velocities within
$H_v$ that do not come arbitrarily close to the scatterers consists of
those contained in horizons for a Lorentz gas of dimension $d_H$ with
infinitesimal scatterers placed at points of ${\cal L}_{inc}$.  Each of
these velocities corresponds to a horizon of dimension $\leq d_H-1$ in $H$
and also in the original space $M$.  However, unlike in the Lorentz
gas with finite sized scatterers, the number of horizons of the highest
dimension $d_H-1$ is now countably infinite, parallel to all possible
$d_H-1$ dimensional subspaces spanned by vectors in ${\cal L}_{inc}$.
These become finite in number as soon as the radius of the scatterers is
increased and the incipient horizon destroyed.
Thus an incipient horizon of dimension $d_H>1$ is in general
associated with an infinite number of horizons of dimension $d_H-1$.
This leads to the interesting effects discussed in Conj.~\ref{c:inc}
above and Sec.~\ref{s:inc} below.

We also remark that incipient horizons contain orbits tangent to
infinitely many scatterers from more than one direction, like
the very recent example of exponential growth of complexity
found in finite horizon multidimensional billiards~\cite{BT}.
Here too, the singularity structure in the vicinity of such orbits
is extremely complicated.  One final definition:

\begin{defn}
$\cal H$ is the set of non-incipient maximal horizons.
\end{defn}

\section{Two lemmas}\label{s:lem}
We now consider the first two propositions in section 2 of Bleher~\cite{B} in the more general
case of arbitrary dimension, which are necessary for the proof of Thm.~\ref{th:F}.  The first
proposition can be naturally generalised, and is given as the first lemma here.  The second
proposition is not valid in $d>2$, and must be replaced by Conj.~\ref{c:hor} above.  The contribution
from each horizon, also needed for Thm.~\ref{th:F}, forms the content of the second lemma.

We generalise Bleher's~\cite{B} Prop. 2.1:
\begin{lem}\label{l:fin}
For any periodic Lorentz gas with infinite horizon there are finitely many maximal (possibly
incipient) horizons.
\end{lem}
\begin{proof}
Assume there are an infinite number of maximal horizons, each with its velocity subspace $V_H$.
The set of $V_H$ is either finite or infinite.  If it is finite, at least one $V_H$ corresponds
to infinitely many $H$, so that $\hat{H}_x$ has infinitely many connected components.  But this
is not possible since the scatterers are comprised of only a finite number of convex pieces.
Thus we conclude there are also an infinite number of $V_H$, leading to unbounded $d_H$-dimensional
covolumes ${\cal V}_H$.  The horizons thus fill at least one space of dimension greater than
$d_H$ more and more densely, which is only possible if the intersection of the interior of the
scatterers with this space is empty.  Thus the space forms its own (possibly incipient) horizon
with a dimension greater than $d_H$, violating the assumption that the original horizons are maximal.  
\end{proof} 

This lemma ensures that the sum over maximal horizons appearing in Thm.~\ref{th:F} is finite.

Bleher's~\cite{B} Prop. 2.2 asserts that in any sufficiently long free path (greater than a
fixed constant $R_0$), all but extreme parts of bounded (less than another constant $R_1$) total
length lie in the space projection $H_x$ of a horizon $H$, or in the terminology of that paper,
a corridor.  For $d>2$ this is not true, since there is in general a complicated set of
horizons of differing dimension connected in a nontrivial manner.

All we can hope for is conjecture~\ref{c:hor} given in the introduction, which we now discuss in
more detail.  The sum over $H\in{\cal H}$ ignores lower dimensional horizons,
some of which are connected to the maximal horizons.  It also 
double counts overlaps of the maximal horizons.  For example,
in the three dimensional Lorentz gas, with a cubic
arrangement of spherical scatterers of radius $r<1/2$ (considered
in more detail later), there are planar horizons parallel to the
$(x,y)$, $(x,z)$ and $(y,z)$ planes, for example the set 
$\{(x,y,z)\in\mathbb{R}^3:r<z<1-r\}$.  Any two of these intersect
in cylindrical sets parallel to one of
the coordinate axes.  There are also lower dimensional (ie cylindrical) 
horizons that are not part
of the planar horizons, for example the set 
$\{(x,y,z)\in\mathbb{R}^3:0<x,y<r, x^2+y^2>r^2 \}$.
Conjecture~\ref{c:hor} asserts that these lower dimensional effects do
not affect the leading order free flight function $F(t)$ if there is a
non-incipient maximal horizon.

The exclusion of purely incipient maximal horizons is necessary, as in
this case $F_H(t)=0$. For $d_H>1$ this is associated with an infinite
number of horizons of lower dimension, whose contributions (including intersections)
would need to be carefully enumerated to calculate $F(t)$; we give some
heuristic arguments and numerical simulations for the case of incipient
principal horizons in Sec.~\ref{s:inc}.

We now come to the most important result of this section:
\begin{lem}\label{l:horsum}
For a non-incipient horizon $H$, we have as $t\to\infty$
\begin{equation}\label{e:gen}
F_H(t)\sim\frac{S_{d_H-1}
\int_{H_x^\perp}\int_{H_x^\perp}
\Delta^{\rm vis}({\bf x},{\bf y})d{\bf x}d{\bf y}}
{S_{d-1}{\cal V}_H^\perp(1-{\cal P})t^{d-d_H}}
\end{equation}
\end{lem}

\begin{proof}
We split $\bf v$ into components in $V_H^\perp$ and
$V_H$ subspaces: ${\bf v}={\bf v}^\perp+{\bf v}^\parallel$, and
similarly for $\bf x$.
The trajectories not reaching the boundary of the horizon before time
$t$ are given by ${\bf v}^\perp\in (\tilde H_x^{\rm vis}({\bf x})-{\bf x})/t$
to leading order in $t$.  Here, $\tilde H_x^{\rm vis}({\bf x})$ is the set
of points ``visible'' from $\bf x$ in $H_x^\perp$, ie the set connected
by a straight line that does not leave $H_x^\perp$. The tilde indicates
that this set is defined on the full $V_H^\perp$, not modulo the lattice,
ie more than one lattice translate of a point may be visible to a given
${\bf x}\in H_x^\perp$.

The definition of $F_H(t)$ is given in Eq.~(\ref{e:F_H}).
Now $\mu$ is just given by $m/[{\cal V}S_{d-1}(1-{\cal P})]$ where $\cal P$ is
the packing fraction of scatterers and $m$ is Lebesgue measure.  We 
integrate over all ${\bf x}^\parallel$ which is the covolume of ${\cal L}_H$ in
$V_H$, ${\cal V}_H$.  The measure over
${\bf x}^\perp$ and ${\bf v}^\perp$, multiplied by $t^{d-d_H}$, is
given by the integral
\[ \int_{H_x^\perp}|\tilde H_x^{\rm vis}({\bf x})|d{\bf x} \]
which can be written more symmetrically as
\[ \int_{H_x^\perp}\int_{H_x^\perp}\Delta^{\rm vis}
({\bf x},{\bf y})d{\bf x}d{\bf y} \]
Since ${\bf v}^\perp$ is of order $t^{-1}$, ${\bf v}^\parallel$
remains of unit magnitude to leading order, and its measure is simply 
$S_{d_H-1}$.  Taking account of the normalisation factors and the relation
${\cal V}={\cal V}_H{\cal V}_H^\perp$ we get the result
stated in the lemma.
\end{proof}

\section{Principal horizons and the free flight function}\label{s:princ}

In the case of a principal horizon, we note that $H_x^\perp$ is
one dimensional, so that the double integral in Lemma 2 is simply 
$|H_x^\perp|^2$.  Otherwise $H_x^\perp$ is typically non-convex and the integral
is not easy to evaluate (an example is provided in the next section).    

Possible principal horizons can be enumerated using the dual lattice
${\cal L}^*=\{{\bf u}\in\mathbb{R}^d:{\bf u}\cdot{\bf v}\in\mathbb{Z}^d,\quad\forall{\bf v}\in{\cal L}\}$
where ${\bf u}\cdot{\bf v}$ denotes the usual inner product in $\mathbb{R}^d$.  In particular,
for each lattice ${\cal L}_H$ of dimension $d-1$, either of the two vectors perpendicular to
${\cal V}_H$ and of length $({\cal V}_H^\perp)^{-1}$ lie in ${\cal L}^*$.  It is not hard to see that
there is a 1:1 correspondence between codimension-1 lattices and primitive elements of ${\cal L}^*$
modulo inversion, where primitive indicates that it is not a multiple of an element of ${\cal L}^*$
by an integer other than $\pm 1$, and inversion is multiplication by $-1$.  Note that the zero
vector is not primitive.  We denote such a vector by ${\bf l}={\bf n}_H/{\cal V}_H^\perp$ with
magnitude $L=|{\bf l}|=({\cal V}_H^\perp)^{-1}$ so that ${\bf n}_H$ is a unit vector perpendicular to
$V_H$.  The subscript $H$ is suppressed for the vectors ${\bf l}$ and their magnitudes $L$.
For each principal horizon there is an ambiguity in the sign of ${\bf l}$ and ${\bf n}_H$
that does not affect any of the results.  When summing over all primitive lattice vectors, the set
denoted $\hat{\cal L}^*$, we need to divide by 2 to take account of this inversion symmetry.

Given a vector ${\bf l}\in\hat{\cal L}^*$, a principal horizon will exist if the perpendicular
projection of the scatterers onto the line parallel to $\bf l$ does not cover the whole of that line.
If there is one spherical
scatterer of radius $r$ per unit cell, this corresponds to the condition that the
width of the horizon $w_H={\cal V}_H^\perp-2r=L^{-1}-2r$ is non-negative.  
Thus for a given $r$, there are a finite number of principal horizons corresponding to vectors
in $\hat{\cal L}^*$ modulo inversion with lengths less than $(2r)^{-1}$; this number grows without
bound as $r$ approaches zero.  For a finite number of arbitrary convex scatterers, this argument
gives an upper bound if $2r$ is interpreted as the smallest width of any of the scatterers.

Returning to the case of a single spherical scatterer of radius
$r$ per unit cell, we have ${\cal P}=V_dr^d$ (see Tab.~\ref{t:Sd}).
Using Thm.~\ref{th:F}, we find that
\begin{equation}\label{e:sum}
F(t)\sim\frac{G_d}{t(1-V_dr^d)}\sum_{{\bf l}\in\hat{\cal L}^*}
Lu_2(L^{-1}-2r)
\end{equation}
where $u_2(x)=x^2u(x)$, $u(x)$ the unit step function, equal to $1$
for positive argument and zero otherwise, and $L=|{\bf l}|$ as above. The geometrical factor 
$G_d=S_{d-2}/(2S_{d-1})$ is given in Tab.~1, where the 2 comes from the inversion symmetry.

Numerical simulations were performed using the author's C++ code, which is 
designed to simulate arbitrary piecewise quadric billiards in arbitrary 
dimensions.  The flight time is thus given the solution of a quadratic
equation; simplifications for the cases where the matrix defining the
quadratic form is zero or the unit matrix are used for greater efficiency.
The main numerical difficulty with billiards is to ensure that the
particle remains on the correct side of the boundary, while retaining
as much accuracy as possible for almost tangential collisions.  These
can be addressed by disallowing very rapid multiple collisions within
a very small perpendicular distance of the boundary.  Tests were made
comparing 16 and 19 digit precision calculations, and by an explicit
investigation of almost tangential collisions.

For the calculations presented here, the spherical scatterer is contained
in a cubic cell.  On reaching the boundary, the particle is transferred to
the opposite side, updating the local position while keeping track of the
overall displacement.  The main obstacle to simulations in dimensions
above $d=10$ considered here is that the volume fraction of balls decreases 
rapidly with dimension for the cubic lattice.  Other lattices (for example
the close packed $E_8$ lattice) may be considered in the future, however
a higher volume fraction comes at the cost of a more complicated unit cell.
It is difficult to get an intuitive grasp of the structure of horizons
of non-cubic lattices above three dimensions; one approach would be
to numerically infer them from $F(t)$ via Thm.~\ref{th:F}.

Numerical simulations of $F(t)$ were performed for the case of a cubic
lattice with $r=0.4$ where there is only a single contribution in the above
sum, for $2\leq d\leq 10$ in Fig.~\ref{f:c4}.  Agreement is good
for low dimensions, while for higher dimensions it is plausible but
limited by finite statistics together with longer times required for
convergence to the final $C/t$ form.  In the case of $r=0.6$ the scatterers
are overlapping, and the maximal horizons have dimension $d-2$.  Here we
expect $C/t^2$ behaviour for $d\geq 3$ but do not have an explicit
analytic formula; see Fig.~\ref{f:c6} which shows good qualitative
agreement for small dimensions, and slow convergence for higher dimensions.
The intermediate case of $r=0.5$ (incipient principal horizon) is considered
in section~\ref{s:inc} below.

\begin{figure}
\includegraphics[width=450pt, height=250pt]{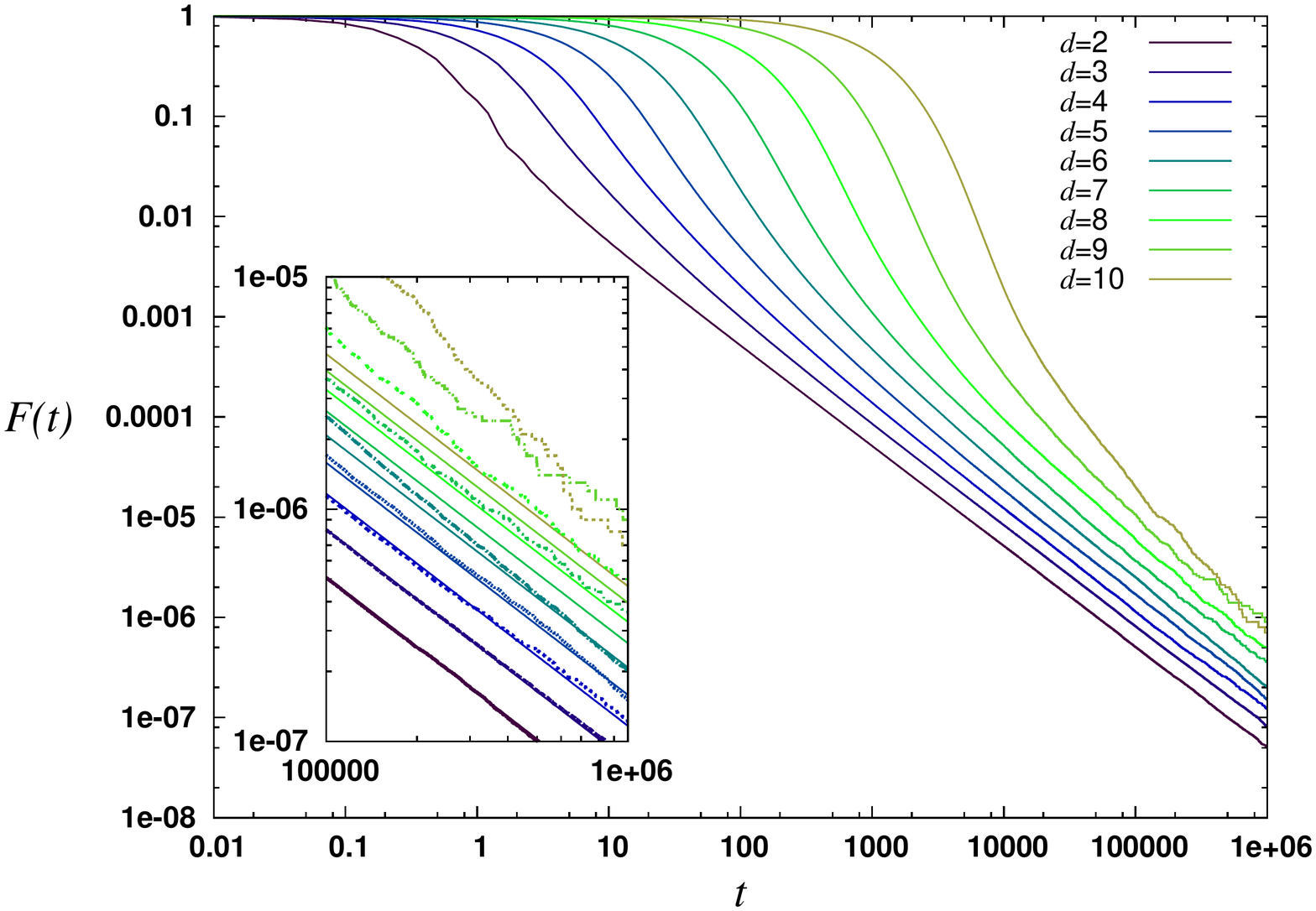}
\caption{\label{f:c4} The free flight function for dimensions $2\leq d\leq 10$ for $r=0.4$ showing $C/t$ behaviour.  The inset shows comparison with the predictions of Eq.~(\protect\ref{e:sum}); here
the numerical results are dotted and the predictions the solid straight lines.}
\end{figure}

\begin{figure}
\includegraphics[width=450pt, height=250pt]{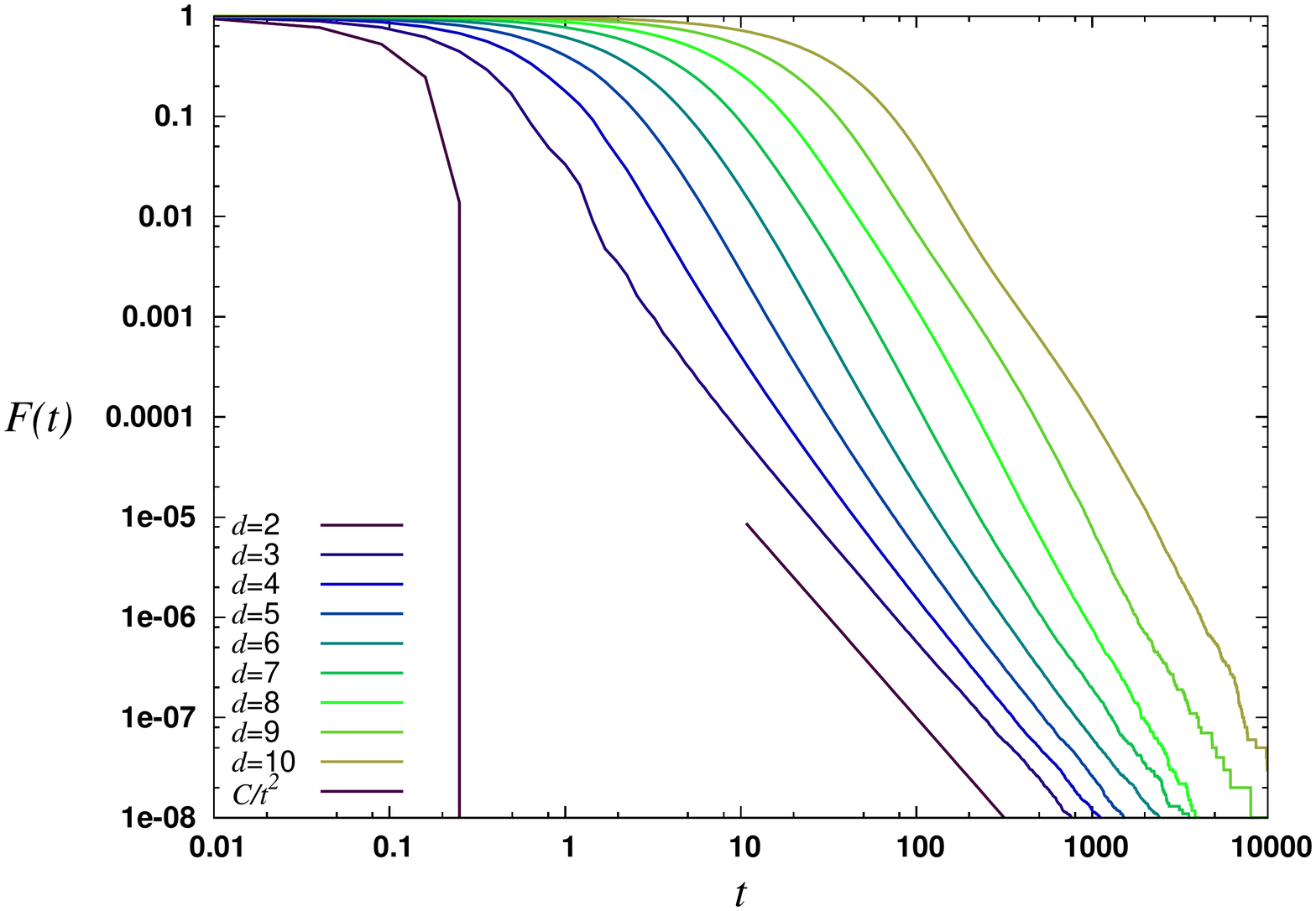}
\caption{\label{f:c6} The free flight function for dimensions $2\leq d\leq 10$ for
$r=0.6$ (overlapping) 
scatterers, showing a collision at bounded time for $d=2$ and approach to $C/t^2$ for $d>2$; the rate of
convergence with $t$  is slow for the higher dimensions considered.}
\end{figure}

We now seek the small $r$ asymptotics of the above expression.  First we 
remove the primitivity condition using M\"obius inversion,
\begin{equation}
F(t)=\frac{G_d}{t(1-V_dr^d)}\sum_{C=1}^\infty\sum'_{{\bf l}\in{\cal L}^*}
\mu(C)CLu_2((CL)^{-1}-2r)
\end{equation}
where the prime indicates that the sum excludes the zero vector.  We
take out the geometric factor and use a Mellin transform on the rest to 
extract the small $r$ behaviour
\begin{equation}\label{e:contour}
F(t)=\frac{G_d}{t(1-V_dr^d)}\frac{1}{2\pi i}\int_{c-i\infty}^{c+i\infty}
\frac{4r^{-s}E((s+1)/2;{\cal L}^*)ds}{2^{s+1}s(s+1)(s+2)\zeta(s+1)}
\end{equation}
where $E(s';{\cal L}^*)$ is the Epstein zeta function for the dual lattice,
\begin{equation}
E(s';{\cal L}^*)=\sum'_{{\bf l}\in{\cal L}^*}L^{-2s'}
\end{equation}

The purpose of the Mellin transform is to express the contour integral as
a sum of residues by Cauchy's theorem, closing the contour to the left.
This sum over poles directly gives an expansion in powers of $r$ (also
logarithms if the poles are not single).  It naturally depends on being
able to show that there exists a sequence of roughly semicircular contours
for negative real part of $s$ and avoiding poles of the integrand,
with integral converging to zero; we are unable to prove this here.

The integrand includes the Riemann zeta function $\zeta(s)$ in the denominator,
so the powers of $r$ appearing in the residue sum depend on the locations of
the zeros of $\zeta(s)$.  The celebrated Riemann Hypothesis~\cite{Co}
asserts that the complex zeros lie on the line with real part of $s$
equal to $1/2$; the zeros on the real axis are all known.  

We also need some analytic properties of the Epstein zeta function in the
numerator; see~\cite{SS}. The function $E(z,{\cal L}^*)$ for arbitrary lattice
of full dimension ${\cal L}^*$ has an analytic
continuation in the whole complex plane except for a single pole at $z=d/2$
with residue $S_{d-1}/2$, a special value $E(0;{\cal L}^*)=-1$ and
zeros at all negative integers.  The following statements~\cite{St}
show how it generalises the Riemann zeta function $\zeta(s)$:
Considering the (self-dual) cubic lattice ${\cal L}={\cal L}^*=\mathbb{Z}^d$, we have
$E(s';\mathbb{Z})=2\zeta(2s')$ and $E(s';\mathbb{Z}^2)=4\zeta(s')\beta(s')$
where $\beta$ is the Dirichlet $\beta$-function, that is, the nontrivial
Dirichlet $L$-function modulo 4.
In four dimensions we have $E(s';{\mathbb Z}^4)=8(1-2^{2-2s'})\zeta(s')\zeta(s'-1)$,
which is already enough to show that the Epstein zeta functions do not
all satisfy a Riemann hypothesis with all complex zeros on a
single vertical line in the complex plane; note that Ref.~\cite{St}
has $2^{2-s'}$ (a misprint).  There is no factorisation for the
physically interesting case of three dimensions, as for example there
are lattice vectors with lengths $\sqrt{3}$ and $\sqrt{5}$ but not
the product $\sqrt{15}$.

The above integral has poles at $s=d-1,-1,-2$ and at complex values
with real part $-1/2$ due to the $\zeta(s+1)$ in the denominator (also
elsewhere if the Riemann Hypothesis is false). Assuming that the contour
can be continued to the left, avoiding poles of the Riemann zeta function,
we find for the $d$-dimensional cubic lattice,
\begin{equation}\label{e:lead}
tF(t)=\frac{\pi^{(d-1)/2}}{2^dd\Gamma((d+3)/2)\zeta(d)r^{d-1}}+O(r^{1/2-\delta})
\end{equation}
where the leading coefficient is the same as found for the Boltzmann-Grad
limit in~\cite{MS2}, and the next term is valid for all $\delta>0$ if the
Riemann Hypothesis is true.  Numerically this term is small in magnitude for
$r$ not too small; the next term
\[ \left[\frac{S_{d-2}V_d}{2^{d-1}(d^3-d)\zeta(d)}-8G_d\right] r \]
is quite visible; see Fig.~\ref{f:sub}.  This argument with the contour
is similar to that of the survival probability of the circular billiard, Ref.~\cite{BD}, and as in that case suggests an equivalence of
Eq.~\ref{e:lead} to the Riemann Hypothesis, requiring a further careful
argument.

\begin{figure}
\includegraphics[width=450pt, height=250pt]{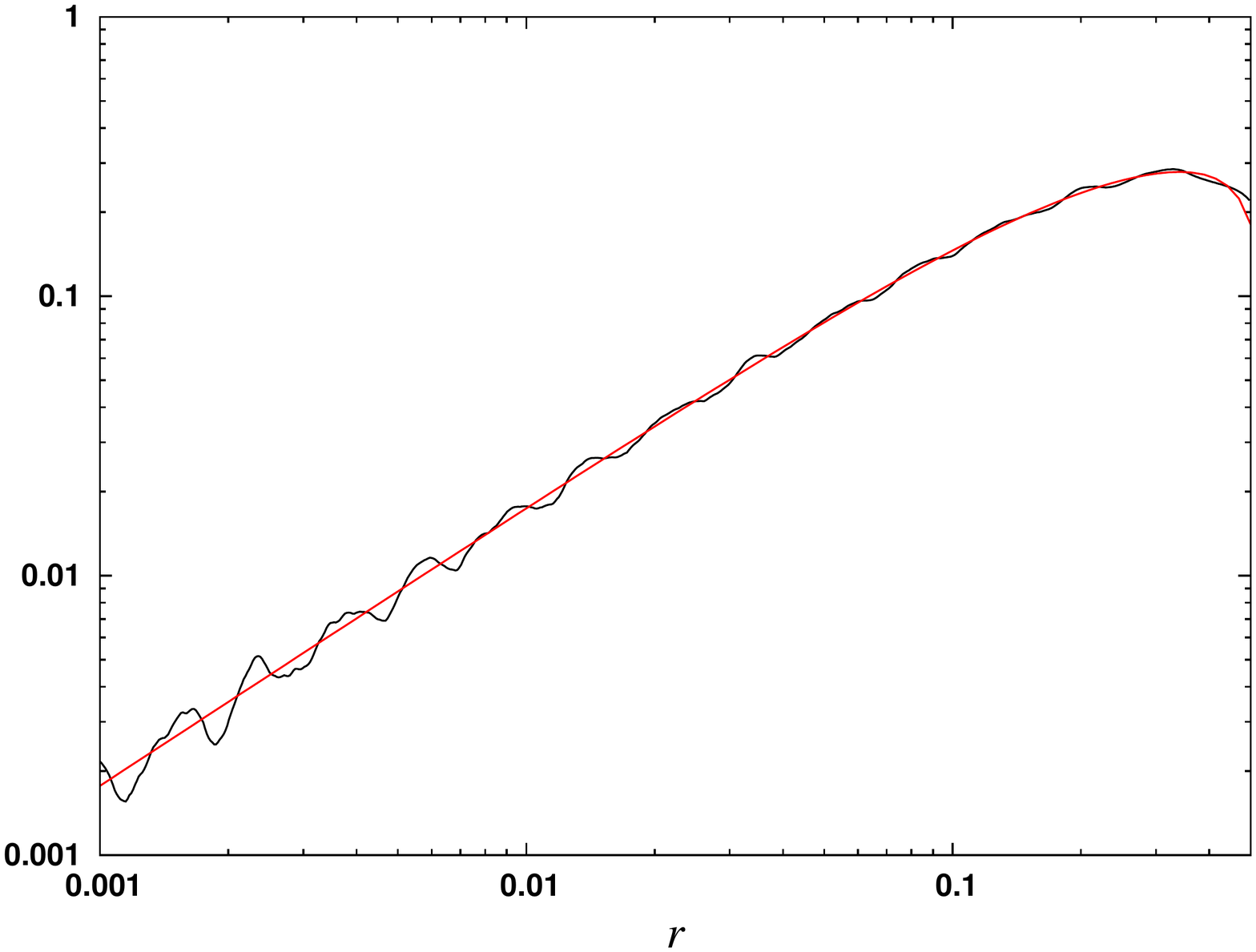}
\caption{\label{f:sub} The sum in Eq.(\protect\ref{e:sum}) minus the leading term in
Eq.~(\protect\ref{e:lead}), compared with the series obtained by summing the residues of poles on the real axis in Eq.~(\protect\ref{e:contour}), the smooth curve.  The smooth terms give a good approximation for the range of $r$ considered;
for smaller $r$ the complex poles coming from the zeta function dominate.}
\end{figure}

In the case of two or more scatterers per unit cell with rationally located
centres and common denominator $q$, the same calculation can be performed.
The Mellin transform then gives an expression of the form
\begin{equation}
tF(t)=\sum_\chi f_\chi(s)\sum'_{{\bf l}\in{\cal L}^*}
\frac{\chi({\bf l}\cdot{\bf m})}{L^{s+1}}
\sum_{C=1}^\infty \frac{\mu(C)\chi(C)}{C^{s+1}}
\end{equation}
where $\chi$ is a Dirichlet character mod $q$, $f_\chi(s)$ and $\bf m$
are explicitly known, the sum over $C$ can be expressed in terms of
known Dirichlet $L$-functions, relating transport in this Lorentz
gas to the generalised Riemann
Hypothesis. The sum over ${\bf l}$ gives Dirichlet-twisted
Epstein zeta-functions.  The latter may be difficult to
characterise in $d=3$ due to the lack of a factorisation of the
untwisted Epstein zeta function.

\section{Incipient horizons and the critical dimension}\label{s:inc}
The results of the last section do not apply directly to the case of
incipient principal horizons, as their width is zero, and they are
associated with an infinite number of horizons of the next lower
dimension.  There is however at least one useful piece of information
to be gained, which we now describe.

Lemma~\ref{l:horsum} shows that the contribution to $F(t)$ from a horizon
of dimension $d-2$ is of order $t^{-2}$.  A finite number of these
contributions are thus of the same order, however a more careful calculation
is required for the present case, where there are an infinite number
of contributions.

We consider a single example, but one which is likely to have more general implications.  Consider the Lorentz gas with a cubic lattice
of arbitrary dimension $d>2$ and $r=1/2$.  The scatterers just touch, leading
to incipient horizons in the coordinate planes and no other principal horizons.
In one of these planes, for example at $x_d=1/2$ for the $d$-th coordinate,
the cross-section of the scatterers are reduced to the points of $\mathbb{Z}^{d-1}$ in the other $d-1$ coordinates,
equivalent to the problem in the previous section, with dimension $d-1$ and $r=0$.
Thus there are horizons of dimension $d-2$ corresponding to all non-zero
lattice vectors ${\bf l}\in\hat{\mathbb{Z}}^{d-1}$ modulo inversion.
For any given lattice vector, there will also be a finite
extent in the $d$-th coordinate direction, which we now quantify.

The space $H_x^\perp$ is two dimensional, in the plane spanned by
$\bf l$ and ${\bf e}_d$, the unit vector in the $d$-th coordinate.
Projected orthogonally into this plane, the lattice has basis vectors
${\bf l}/L^2$ (hence of magnitude $L^{-1}$, where $L=|{\bf l}|$, consistent
with previous notation) and ${\bf e}_d$. This means $H_x^\perp$ consists of
the area confined between four circles of radius $r=1/2$
and centred at $(0,0), (L^{-1},0), (0,1), (L^{-1},1)$ with respect to the
orthonormal basis $\{{\bf l}/L,{\bf e}_d\}$; we will denote coordinates
with respect to this basis as $(x,y)$. Pairs of scatterers separated by
$L^{-1}$ overlap (except when $L=1$) and pairs displaced by $1$ just touch.
The region not within radius $r=1/2$ of the centre of any scatterer is
contained in the rectangle defined by $0<x<L^{-1}$ and
$\frac{1}{2}\sqrt{1-L^{-2}}<y<1-\frac{1}{2}\sqrt{1-L^{-2}}$.  In the limit
as $L\to\infty$ the height in the $y$ direction is $L^{-2}/2$ to leading order.
In this limit, we scale in both directions, $x-L^{-1}/2=L^{-1}\xi$ and
$y-1/2=L^{-2}\eta$ and keep only the highest order nontrivial terms
in the equation, and find that $H_x^\perp$ is defined by a limiting
shape bounded by four parabolas,
\begin{equation}\label{e:dia}
|\eta|<(|\xi|-1/2)^2,\qquad |\xi|<1/2
\end{equation}
This region is plotted in Fig.~\ref{f:dia}.
There is thus a limiting nonzero probability for two points in the
region to be mutually ``visible'', ie for which the line joining the
points remains in the region.  

\begin{figure}
\includegraphics[width=450pt, height=250pt]{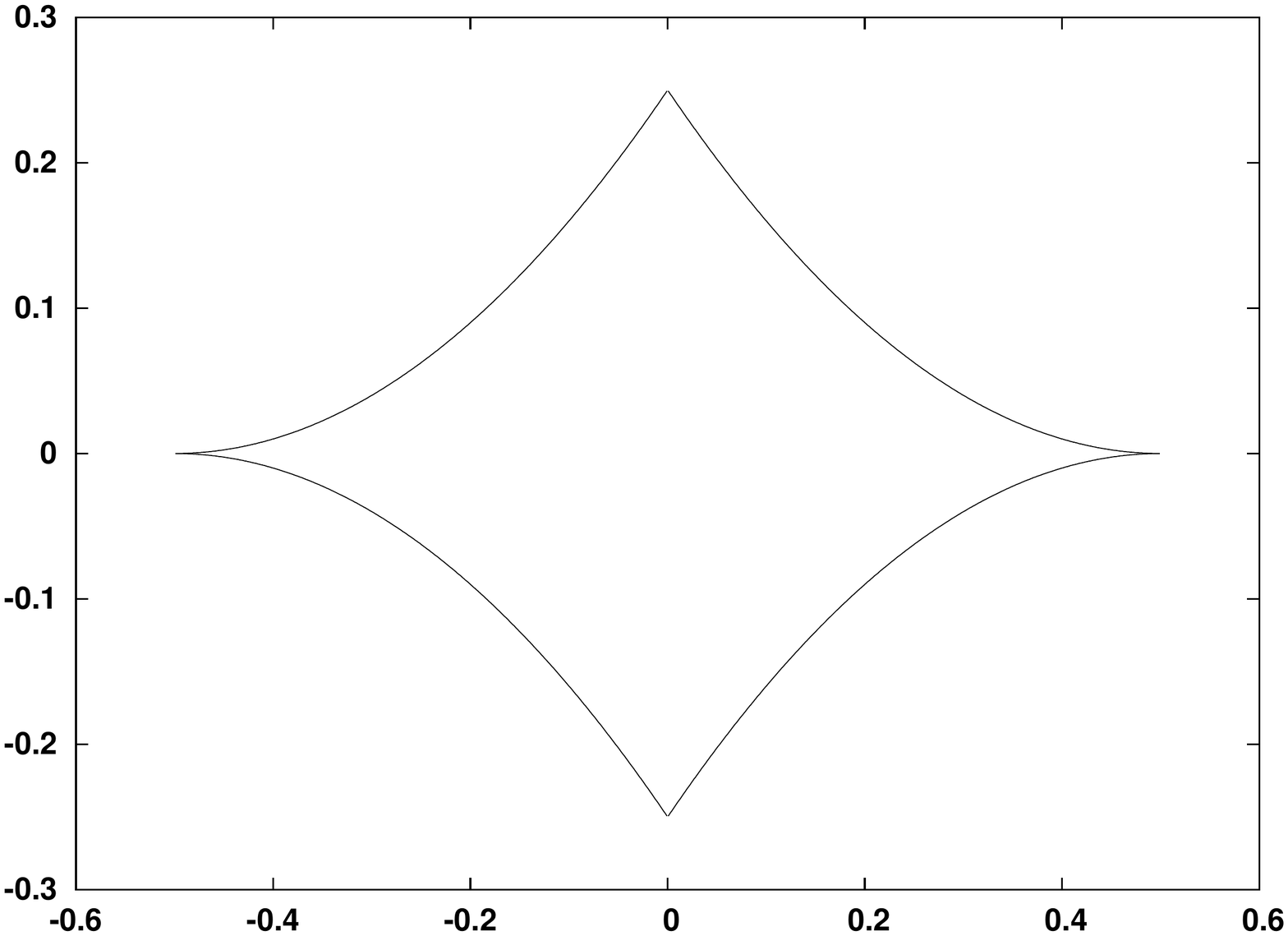}
\caption{\label{f:dia} The region defined by Eq.~\protect\ref{e:dia}.}
\end{figure}

The double integral over $\int_{H_x^\perp}$ in Lemma 2 is thus
asymptotic to a constant (numerically found to be approximately
$0.02746$) multiplied by $L^{-6}$ as $L\to\infty$.
The extra factor of $L$ in Lemma~\ref{l:horsum} shows that
\begin{equation}
F_H(t)\sim \frac{C_d}{t^2L^5}
\end{equation}
for some constant $C_d$ containing only geometrical factors, in
the appropriate limit (first $t\to\infty$, then $L\to\infty$).
Summing over ${\bf l}\in\hat{\mathbb{Z}}^{d-1}$ gives the Epstein
zeta function $E(5/2;\mathbb{Z}^{d-1})/\zeta(5)$.  This is a convergent series
for $d\leq 5$ and divergent for $d\geq 6$.  Thus $d=6$ is a critical
dimension for these incipient horizons.  The divergence of the sum
indicates that $F(t)$ has a slower decay with $t$ than $t^{-2}$ but
a detailed calculation is required, taking into account the overlap
of the relevant horizons, to yield a precise form.  It is unlikely
that the result would reach that of the non-incipient principal
horizon, that is, $t^{-1}$; we propose Conj.~\ref{c:inc} above.

Numerical tests of this conjecture are somewhat difficult due to the
low packing fraction in high dimensions (so that individual
trajectories survive for long times) and slow convergence to the long
time limit.  Scatterer radii $r=0.4$ and $r=0.6$ were discussed previously;
see Figs.~\ref{f:c4} and~\ref{f:c6}.   For a scatterer of radius $0.5$ and
dimension $d\geq 6$ the exponent itself is unknown; results are shown in
Fig.~\ref{f:c5} with fitted exponent in Tab.~2, and show exponents greater
than $2$ for low dimensions due to the finite times involved.  

\begin{figure}
\includegraphics[width=450pt, height=250pt]{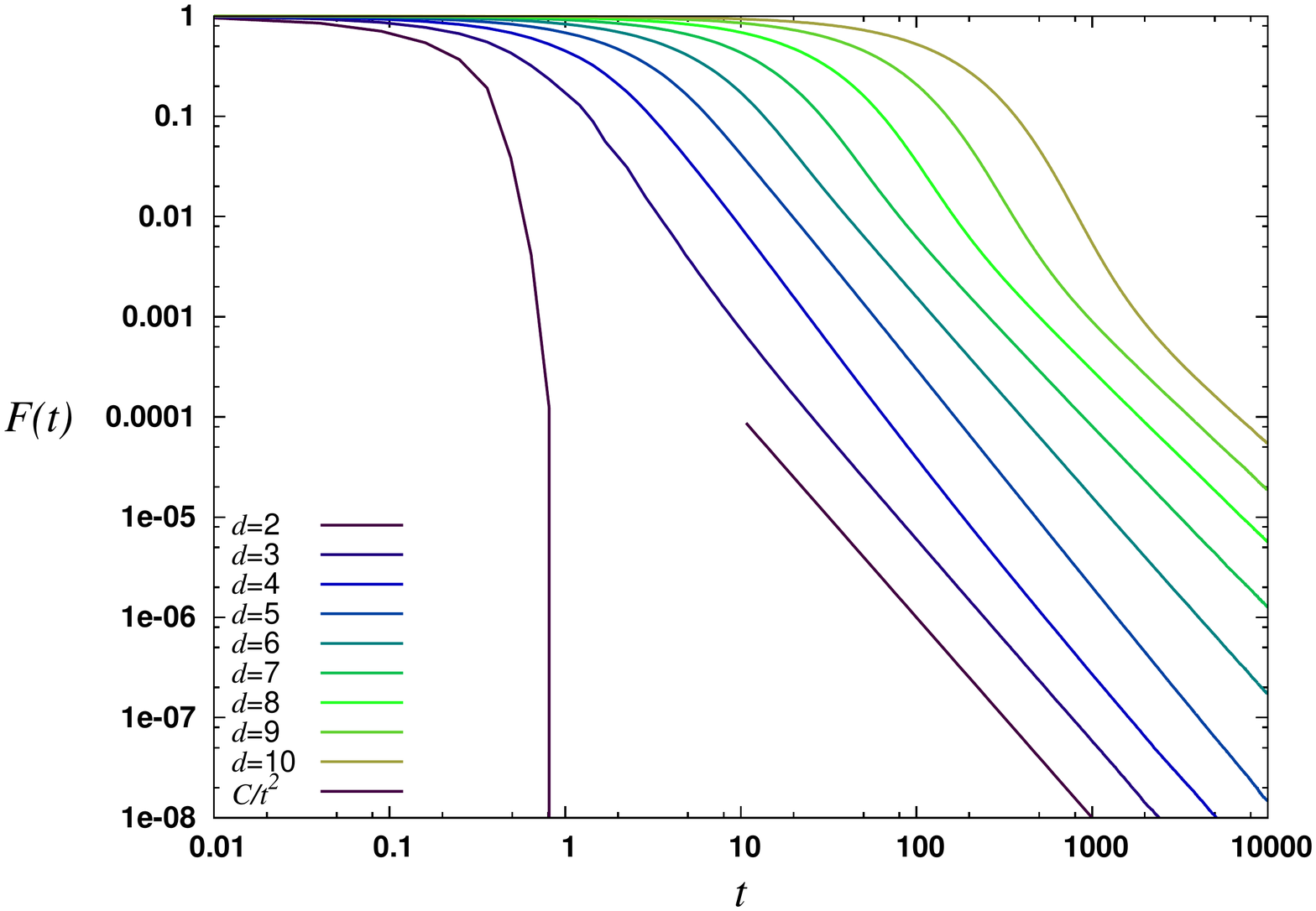}
\caption{\label{f:c5} The free flight survival probability for dimensions
$2\leq d\leq 10$ for $r=0.5$ (touching) 
scatterers, showing escape at finite time for $d=2$, approach to $C/t^2$ for $2<d<5$ and slower
decay for $d\geq 6$ as predicted by theory; fits for the relevant exponents are given in
Tab.~\protect\ref{t:exp}.}
\end{figure}

\begin{table}
\centerline{
\begin{tabular}{|c|ccccccccc|}\hline
$d$&3&4&5&6&7&8&9&10&\\\hline
$n=10^8$&2.07&2.17&2.19&1.99&1.82&1.70&1.64&1.58&\\
$n=10^9$&2.02&2.20&2.17&1.98&1.80&1.67&&&\\
$n=10^{10}$&2.03&2.08&2.14&1.95&&&&&\\\hline
\end{tabular}
}
\caption{\label{t:exp} Numerical estimate of the exponent $\alpha$ from Conjecture~\ref{c:inc},
for cubic Lorentz gases with $r=1/2$, sample size $n$, and fitted
for free flight survival probabilities $10^2/n<F(t)<10^4/n$.}
\end{table}

\section{Correlations and diffusion}\label{s:diff}
The above expression for $F(t)$ is also relevant to correlation functions,
since conjecture~\ref{c:corr} asserts that
following a long flight, correlations decay very rapidly.  For
describing diffusion, we need velocity correlation functions.
Figs.~\ref{f:fit} and~\ref{f:corr} show the velocity autocorrelation function
(used below) compared with the sum of horizons predicted by
Conj.~\ref{c:corr} in the case of a three dimensional cubic
Lorentz gas.

\begin{figure}
\centerline{\includegraphics[width=350pt]{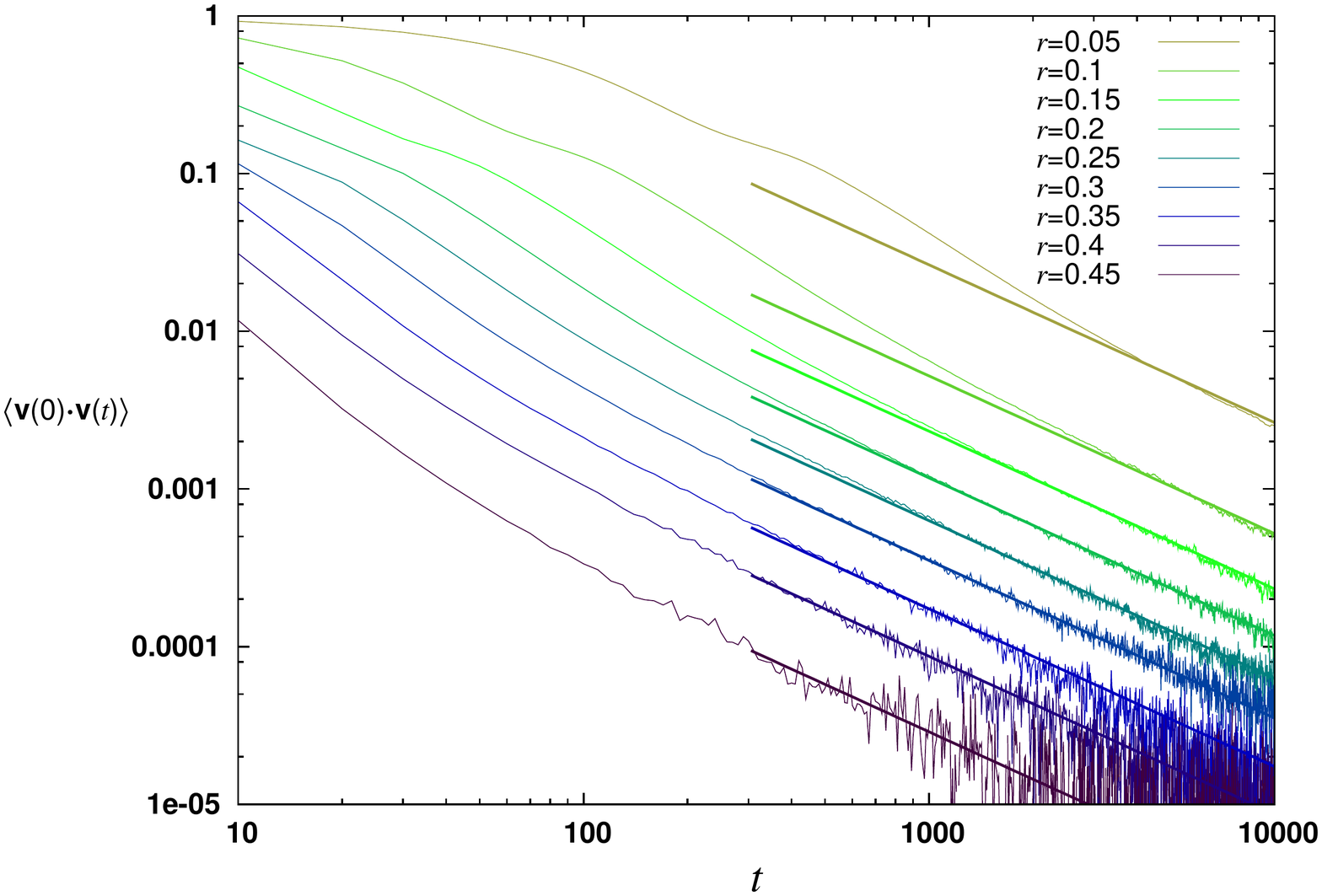}}
\caption{\label{f:fit} For $r<0.5$, Eq.~\protect\ref{e:sum} gives the
limit of $tF(t)$ as $t\to\infty$ as a sum over horizons, assuming
Conj.~\protect\ref{c:hor}. Conj.~\ref{c:corr} applied to the components
of the velocity asserts that the same limit is given by
$t\langle{\bf v}_0\cdot{\bf v}_t\rangle$ related to the
diffusion by Eq.~(\ref{e:GK}). Here this correlation function
is shown with a fit to the form $C_r/t$, for the three dimensional cubic
Lorentz gas.}
\end{figure}

\begin{figure}
\centerline{\includegraphics[width=350pt]{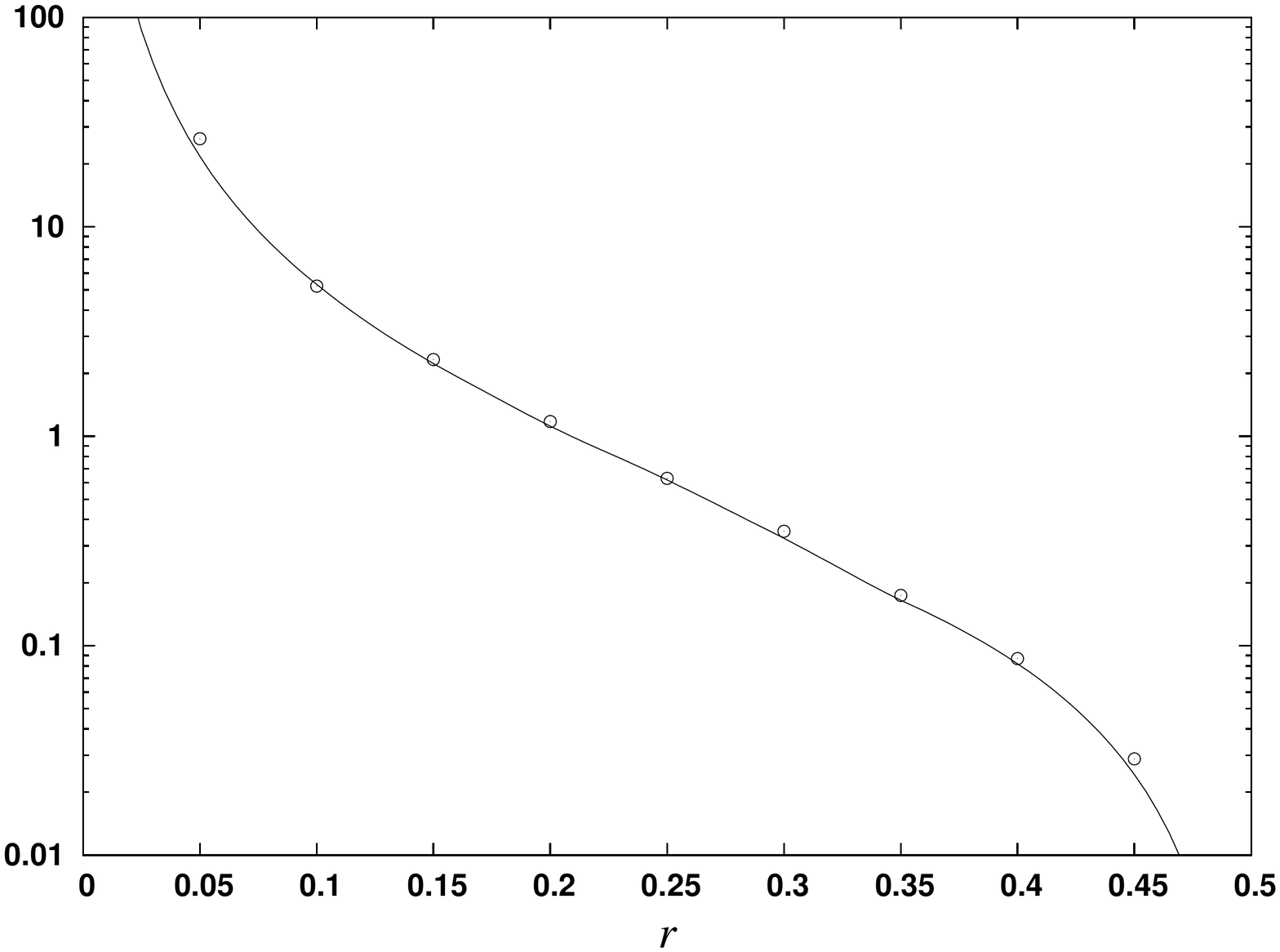}}
\caption{\label{f:corr} See Fig.~\protect\ref{f:fit}: Comparison of the
fitted $C_r$ (single points) with $\lim_{t\to\infty} t F(t)$ from Eq.~\ref{e:sum}
(smooth curve).}
\end{figure}

Recalling the discussion in previous sections, we are interested in
the case of nonincipient principal horizons for which the function
$F(t)$ decays as $C/t$; only in this case do we expect Eq.~(\ref{e:GK})
to give superdiffusion, ie a nonzero value of $\cal D$.  As discussed
for $F(t)$, however, the coefficient is explicit.
Now we give the proof of Thm.~\ref{th:D}

\begin{proof}
Assuming conjecture~\ref{c:corr}, the velocity correlations are determined to
leading order by the integral over orbits that do not reach a scatterer
in the given time, that is, $M_t$.  For these orbits, $\bf v$ is a
constant, so conjecture~\ref{c:hor} can be applied.  We then
use the same method as in the proof of Lemma 2, except that the integrand
varies significantly over the ${\bf v}^{\parallel}$ space, ie $H_v$.  Given
that the perpendicular space is one-dimensional, we find a simple result in
terms of the width of the horizon $w_H$ defined at the beginning
of Sec.~\ref{s:princ}:
\begin{equation}
\langle v_i(0)v_j(t) \rangle\sim\sum_{H\in{\cal H}}
\frac{w_H^2\int_{H_v}v_iv_j}
{tS_{d-1}{\cal V}_H^\perp(1-{\cal P})}
\end{equation}
as $t\to\infty$.  Now $H_v$ consists of all unit vectors $\bf v$ perpendicular to a
unit vector $\bf n$, which depends on the horizon $H$ but we will suppress this
in the notation.  This $\bf n$ is parallel to a vector in ${\cal L}^*$ as
discussed previously.
Switching to a coordinate system based on $\bf n$ we can evaluate the
integral to give
\begin{equation}
\int_{H_v}v_iv_j=(\delta_{ij}-n_in_j)V_{d-1}
\end{equation}
The cases $d>2$ and $d=2$ need to be derived separately, but give the
same result.  Putting this together we find the stated result.
\end{proof}

For the case of the cubic lattice, the superdiffusion matrix is
isotropic.  For $1/(2\sqrt{2})<r<1/2$ we have only $d$ horizons
of the form $(1,0,0,\ldots)$ and find that ${\cal D}_{ij}={\cal D}\delta_{ij}$
with
\begin{equation}
{\cal D}=\frac{(1-2r)^2}{1-V_dr^d}
\frac{\Gamma(d/2)}{\sqrt{\pi}\Gamma((d-1)/2)}
\end{equation}
while for $1/(2\sqrt{3})<r<1/(2\sqrt{2})$ there is an additional term
coming from the $d(d-1)$ horizons of the form $(1,\pm 1,0,\ldots)$,
which is thus larger by a factor of order $d$.  In general this means
that the large $d$ asymptotics at fixed $r$ has a different form every
time a new horizon is introduced.

For the limit of small scatterers and fixed $d$ the sum over horizons
is exactly the same as for $F(t)$ in Sec.~\ref{s:princ}, and we can
directly combine Eqs.~(\ref{e:DF},\ref{e:lead}) to get
\begin{equation}
{\cal D}=\frac{\pi^{(d-1)/2}}
{2^{d}d^2\Gamma((d+3)/2)
\zeta(d)r^{d-1}}+O(r^{1/2-\delta})
\end{equation}
for $\delta>0$.  Again it appears that this statement is equivalent to the Riemann
Hypothesis under the same caveat (convergence of a sequence of
contour integrals) as in Sec.~\ref{s:princ}.

We can also compute the diffusion numerically, however even for the
lowest dimensions, normal diffusion unrelated to the horizons dominates
at accessible timescales; see Fig.~\ref{f:diff}.

\begin{figure}
\includegraphics[width=450pt, height=250pt]{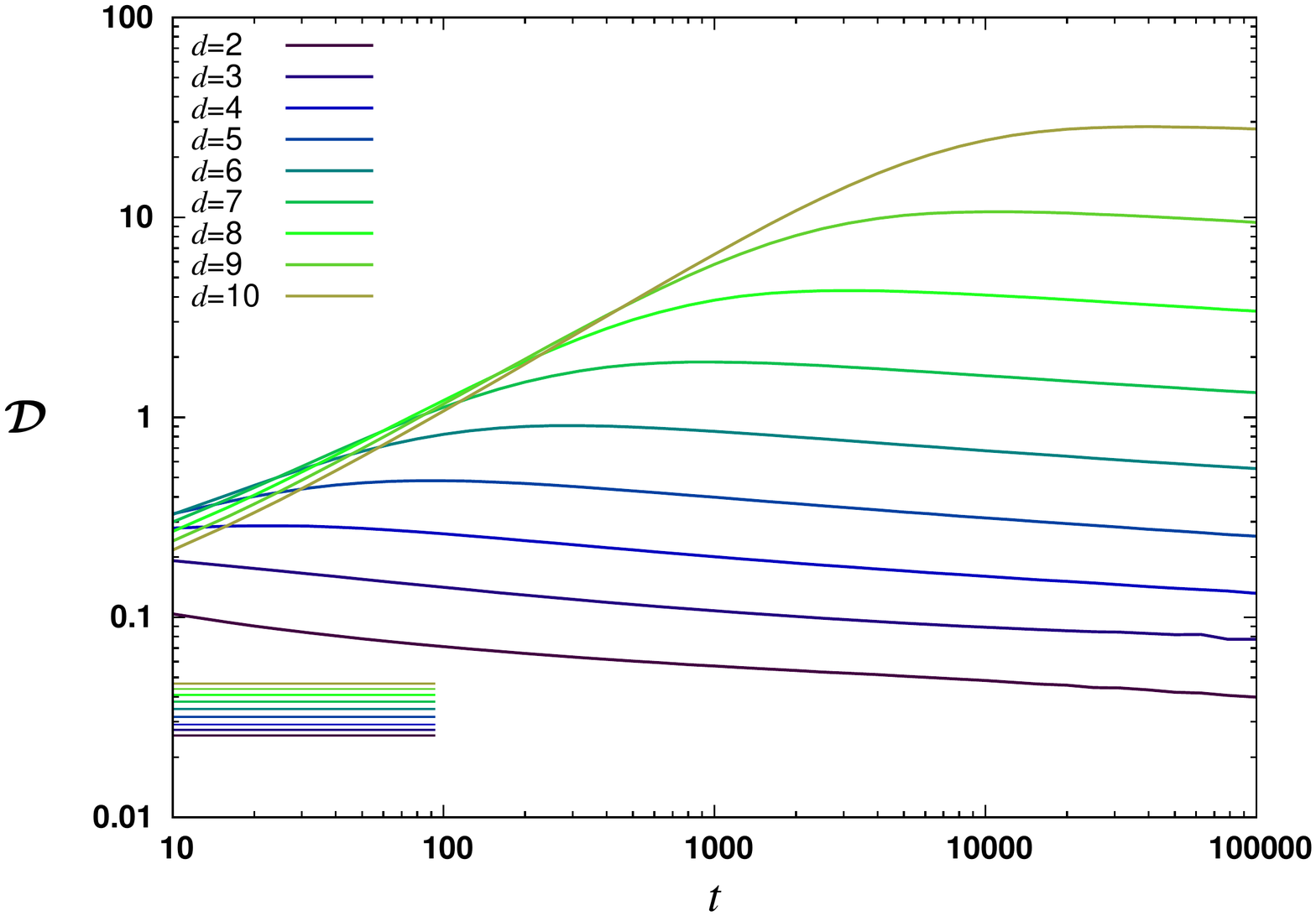}
\caption{\label{f:diff} The superdiffusion coefficient ${\cal D}$ 
numerically (dotted lines) and theoretically (horizontal lines).  At 
accessible timescales normal diffusion processes dominate.}
\end{figure}

\section{Diffusion in discrete time}\label{s:disc}
In order to compare our explicit expressions with other literature, we
note that~\cite{SV} uses exclusively discrete time properties (using $n$
rather than $t$ in the definitions), while~\cite{B,CD} link discrete to
continuous time preperties. Recall the discussion in Sec.~\ref{s:DXi}:
${\cal D}^{\rm disc}$ is formally infinite, however we can discuss the
covariance of the limiting distribution, which is related
to the superdiffusion coefficient by $\Xi={\cal D}$ and to its discrete
time version by $\Xi^{\rm disc}=\tau\Xi$.  
The mean free time $\tau$ is given by~\cite{C}
\begin{equation}
\tau=\frac{1-{\cal P}}{|\partial\Omega\cap E|}\frac{S_{d-1}}{V_{d-1}}
\end{equation}
Here  $|\partial\Omega\cap E|$ is the measure of the boundary of the
scatterers, equal to $S_{d-1}r^{d-1}$ in the case of a single scatterer per unit
cell, without overlapping.  Thus in the latter case we find
\begin{equation}
\tau=\frac{1-V_dr^d}{V_{d-1}r^{d-1}}
\end{equation}
which when combined with Thm.~\ref{th:D} and the above relations gives a simpler prefactor:
\begin{equation}
\Xi^{\rm disc}_{ij}=\frac{1}{r^{d-1}S_{d-1}}\sum_{H\in{\cal H}}\frac{w_H^2}{{\cal V}_H^\perp}
(\delta_{ij}-n_i(H)n_j(H))\;\;.
\end{equation}

\section*{Acknowledgements}
The author thanks Domokos Sz\'asz for suggesting the problem of diffusion
in high dimensional Lorentz gases and for helpful discussions; acknowledges
funding from the Royal Society and the organisers of ``Chaotic and Transport
Properties of Higher-Dimensional Dynamical Systems'' (Cuernavaca, Jan 2011),
where some ideas were further developed; and is also grateful to Leonid 
Bunimovich, Nikolai Chernov, Orestis Georgiou, Jens Marklof, Ian Melbourne,
Zeev Rudnick, David Sanders and Trevor Wooley for helpful discussions.

\end{document}